\documentclass[11pt]{article}
\usepackage{amsfonts,mathrsfs,amssymb,amsthm}
\usepackage{mathtools}
\usepackage{eucal}

\usepackage{amsmath,amscd}
\usepackage{pslatex}
\usepackage{color}
\usepackage{xcolor}

\usepackage[all]{xy}
\usepackage{pdfpages}
\usepackage{tikz-cd}
\usepackage{graphicx}
\usepackage{fancyhdr}
\usepackage{tocloft}

\definecolor{myred}{HTML}{B22222}
\definecolor{mygreen}{HTML}{168C40}

\usepackage{ifpdf}
\ifpdf

\usepackage[colorlinks,final,hyperindex,linkcolor=myred,citecolor=mygreen]{hyperref} 
\else
\usepackage[colorlinks,final,backref=page,hyperindex,hypertex,linkcolor=myred,citecolor=mygreen]{hyperref}
\fi

\begin{document}

\theoremstyle{plain}
\newtheorem{defn}{Definition}[section]
\newtheorem{prop}[defn]{Proposition}
\newtheorem{thm}[defn]{Theorem}
\newtheorem{lem}[defn]{Lemma}
\newtheorem{cor}[defn]{Corollary}

\theoremstyle{definition}
\newtheorem{exam}[defn]{Example}

\theoremstyle{remark}
\newtheorem{rem}[defn]{Remark}

\newcommand{\add}{{\rm add}}
\newcommand{\con}{{\rm con}}
\newcommand{\gd}{{\rm gldim}}
\newcommand{\qgd}{{\rm qgldim}}
\newcommand{\sd}{{\rm stdim}}
\newcommand{\sr}{{\rm sr}}
\newcommand{\dm}{{\rm domdim}}
\newcommand{\cdm}{{\rm codomdim}}
\newcommand{\tdim}{{\rm dim}}
\newcommand{\E}{{\rm E}}
\newcommand{\Mor}{{\rm Morph}}
\newcommand{\End}{{\rm End}}
\newcommand{\ind}{{\rm ind}}
\newcommand{\rsd}{{\rm resdim}}
\newcommand{\rd} {{\rm rd}}
\newcommand{\ol}{\overline}
\newcommand{\overpr}{$\hfill\square$}
\newcommand{\rad}{{\rm rad}}
\newcommand{\soc}{{\rm soc}}
\renewcommand{\top}{{\rm top}}
\newcommand{\pd}{{\rm pd}}
\newcommand{\id}{{\rm idim}}
\newcommand{\fld}{{\rm fdim}}
\newcommand{\Fac}{{\rm Fac}}
\newcommand{\Gen}{{\rm Gen}}
\newcommand{\fd} {{\rm findim}}
\newcommand{\qpd} {{\rm qpd}}
\newcommand{\Fd} {{\rm Findim}}
\newcommand{\Pf}[1]{{\mathscr P}^{<\infty}(#1)}
\newcommand{\DTr}{{\rm DTr}}
\newcommand{\cpx}[1]{#1^{\bullet}}
\newcommand{\D}[1]{{\mathscr D}(#1)}
\newcommand{\Dz}[1]{{\mathscr D}^+(#1)}
\newcommand{\Df}[1]{{\mathscr D}^-(#1)}
\newcommand{\Db}[1]{{\mathscr D}^b(#1)}
\newcommand{\C}[1]{{\mathscr C}(#1)}
\newcommand{\Cz}[1]{{\mathscr C}^+(#1)}
\newcommand{\Cf}[1]{{\mathscr C}^-(#1)}
\newcommand{\Cb}[1]{{\mathscr C}^b(#1)}
\newcommand{\Dc}[1]{{\mathscr D}^c(#1)}
\newcommand{\K}[1]{{\mathscr K}(#1)}
\newcommand{\Kz}[1]{{\mathscr K}^+(#1)}
\newcommand{\Kf}[1]{{\mathscr  K}^-(#1)}
\newcommand{\Kb}[1]{{\mathscr K}^b(#1)}
\newcommand{\DF}[1]{{\mathscr D}_F(#1)}
\newcommand{\Kac}[1]{{\mathscr K}_{\rm ac}(#1)}
\newcommand{\Keac}[1]{{\mathscr K}_{\mbox{\rm e-ac}}(#1)}
\newcommand{\modcat}{\ensuremath{\mbox{{\rm -mod}}}}
\newcommand{\Modcat}{\ensuremath{\mbox{{\rm -Mod}}}}
\newcommand{\Spec}{{\rm Spec}}
\newcommand{\stmc}[1]{#1\mbox{{\rm -{\underline{mod}}}}}
\newcommand{\Stmc}[1]{#1\mbox{{\rm -{\underline{Mod}}}}}
\newcommand{\prj}[1]{#1\mbox{{\rm -proj}}}
\newcommand{\inj}[1]{#1\mbox{{\rm -inj}}}
\newcommand{\Prj}[1]{#1\mbox{{\rm -Proj}}}
\newcommand{\Inj}[1]{#1\mbox{{\rm -Inj}}}
\newcommand{\PI}[1]{#1\mbox{{\rm -Prinj}}}
\newcommand{\GP}[1]{#1\mbox{{\rm -GProj}}}
\newcommand{\GI}[1]{#1\mbox{{\rm -GInj}}}
\newcommand{\gp}[1]{#1\mbox{{\rm -Gproj}}}
\newcommand{\gi}[1]{#1\mbox{{\rm -Ginj}}}
\newcommand{\Pmodcat}[1]{#1\mbox{{\rm -Proj}}}
\newcommand{\opp}{^{\rm op}}
\newcommand{\otimesL}{\mathbin{\overset{\raisebox{-0.4ex}{$\scriptstyle\mathbb{L}$}}{\otimes}}}
\newcommand{\rHom}{{\rm\mathbb R}{\rm Hom}}
\newcommand{\pdim}{\pd}
\newcommand{\Hom}{{\rm Hom}}
\newcommand{\Coker}{{\rm Coker}}
\newcommand{ \Ker  }{{\rm Ker}}
\newcommand{ \Cone }{{\rm Con}}
\newcommand{ \Img  }{{\rm Im}}
\newcommand{\Ext}{{\rm Ext}}
\newcommand{\StHom}{{\rm \underline{Hom}}}
\newcommand{\StEnd}{{\rm \underline{End}}}
\newcommand{\KK}{I\!\!K}
\newcommand{\gm}{{\rm _{\Gamma_M}}}
\newcommand{\gmr}{{\rm _{\Gamma_M^R}}}
\def\vez{\varepsilon}\def\bz{\bigoplus}  \def\sz {\oplus}
\def\epa{\xrightarrow} \def\inja{\hookrightarrow}
\newcommand{\lra}{\longrightarrow}
\newcommand{\llra}{\longleftarrow}
\newcommand{\lraf}[1]{\stackrel{#1}{\lra}}
\newcommand{\llaf}[1]{\stackrel{#1}{\llra}}
\newcommand{\ra}{\rightarrow}
\newcommand{\dk}{{\rm dim_{_{k}}}}
\newcommand{\holim}{{\rm Holim}}
\newcommand{\hocolim}{{\rm Hocolim}}
\newcommand{\colim}{{\rm colim\, }}
\newcommand{\limt}{{\rm lim\, }}
\newcommand{\Add}{{\rm Add }}
\newcommand{\Prod}{{\rm Prod }}
\newcommand{\Tor}{{\rm Tor}}
\newcommand{\Cogen}{{\rm Cogen}}
\newcommand{\Tria}{{\rm Tria}}
\newcommand{\Loc}{{\rm Loc}}
\newcommand{\Coloc}{{\rm Coloc}}
\newcommand{\tria}{{\rm tria}}
\newcommand{\Con}{{\rm Con}}
\newcommand{\Thick}{{\rm Thick}}
\newcommand{\thick}{{\rm thick}}
\newcommand{\Sum}{{\rm Sum}}

\newcommand{\Coprod}{{\rm Coprod}}
\newcommand{\coprods}{{\rm coprod}}

\newcommand{\Tsb}{\mathcal{T}^{sb}}

\newcommand{\T}{\mathcal{T}}

\newcommand{\U}{\mathcal{U}}

\newcommand{\Hocolim}{\underrightarrow{{\rm Hocolim}}}
\newcommand{\Holim}{\underleftarrow{{\rm Holim}}}

\newcommand{\lims}{\underleftarrow{{\rm lim}}}

\newcommand{\colims}{\underrightarrow{{\rm colim}}}

\newcommand{\spec}{{\rm Spec}}

\newcommand{\R}{\mathbb{R}}

\begin{center}
{\Large {\bf Bounded $t$-structures on the category of strongly bounded objects}}
\end{center}

\centerline{\footnotesize{Hongxing Chen$^*$, Xiaohu Chen and Jinbi Zhang}}

\renewcommand{\cftsubsecleader}{\hfill}

\setlength{\cftsubsecindent}{0pt}

\renewcommand{\cftsecfont}{\mdseries\normalsize}       
\renewcommand{\cftsubsecfont}{\mdseries\normalsize}    
\renewcommand{\cftsubsubsecfont}{\mdseries\normalsize} 

\renewcommand{\cftsecpagefont}{\mdseries\normalsize}
\renewcommand{\cftsubsecpagefont}{\mdseries\normalsize}

\renewcommand{\thefootnote}{\alph{footnote}}
\setcounter{footnote}{-1} \footnote{ $^*$ Corresponding author.
Email: chenhx@cnu.edu.cn.}
\renewcommand{\thefootnote}{\alph{footnote}}
\setcounter{footnote}{-1} \footnote{2020 Mathematics Subject
Classification: Primary 18G80; Secondary 18G20, 16E35, 16G10.}
\renewcommand{\thefootnote}{\alph{footnote}}
\setcounter{footnote}{-1} \footnote{Keywords: Bounded $t$-structure; finite strong finitistic dimension; Gorenstein triangulated category; weak approximation.}

\begin{abstract}
Strongly bounded objects in a weakly approximable triangulated category are known to relate closely to global dimension and to play a significant role in the uniqueness problem for triangulated enhancements. In this paper, we investigate when the full subcategory of strongly bounded objects admits a bounded $t$-structure. Our main result, under a finiteness condition termed the finite strong finitistic dimension, states that this happens exactly when the subcategory agrees with the full subcategory of bounded objects in the ambient triangulated category. In that case, the bounded $t$-structure is unique up to equivalence; even more, the uniqueness holds unconditionally on the full subcategory of bounded objects.
\end{abstract}


\section{Introduction}\label{Introduction}

Bounded $t$-structures on triangulated categories were originally introduced by Beilinson, Bernstein and Deligne (see \cite{BBD}) in their foundational work on perverse sheaves on algebraic and analytic varieties. Their existence has been a key prerequisite for the discussion of stability conditions on a triangulated category (see \cite{B07}). In recent years, substantial progress has been made toward
a systematic understanding of when bounded $t$-structures exist on a given triangulated category (see \cite{AGH19,S22,N24A,BCRPZ24}). For example, Neeman proved in \cite{N24A} a conjecture by Antieau, Gepner and Heller regarding the equivalence between the regularity of finite-dimensional noetherian schemes and the existence of bounded $t$-structures on their derived categories of perfect complexes. More generally, for any essentially small triangulated category under a finiteness assumption, a categorical obstruction (the singularity category) to the existence of bounded $t$-structures on it was established in \cite{BCRPZ24}. In the special case of the subcategory of compact
objects in a compactly generated triangulated category with a single compact generator, the existence of
a bounded $t$-structure forces this subcategory to coincide with the category of bounded pseudo-compact
objects.

In this paper, we address a parallel question regarding the existence of bounded $t$-structures on the category of \emph{strongly bounded objects} in a weakly approximable triangulated category. This category, first introduced by Neeman in \cite{N25}, can be regarded as an analogue of the homotopy category of bounded complexes of projective modules over associative rings. In \cite{CCZ26}, it was shown to have strong connections with the global dimension of triangulated categories. Furthermore, as demonstrated in \cite{CNS26}, it plays a significant role in the study of the uniqueness of enhancements of the natural subcategories of weakly approximable triangulated categories.

Let $\T$ be a \emph{compactly generated triangulated category generated by a compact generator $G$}. Associated with $G$, there is a canonical $t$-structure~$(\T_{G}^{\leq 0}, \T_{G}^{\geq 1})$ (see Lemma~\ref{lem-p-w-gen}) on $\T$ and its equivalence class is called the \emph{preferred equivalence class} in \cite[Definition~1.18]{N18}. We are interested in the following triangulated subcategories of $\T$ that are independent of the choices of compact generators of $\T$ and $t$-structures in the preferred equivalence class (see \cite{CCZ26,N18,N25}).

\vspace{0.28em}

$\bullet$ \emph{Bounded above objects}: $\T^-\coloneqq \bigcup_{n\in\mathbb{N}}\T_{G}^{\leq n}$;

\vspace{0.28em}

$\bullet$ \emph{Bounded below objects}: $\T^+\coloneqq \bigcup_{n\in\mathbb{N}}\T_{G}^{\geq-n}$;

\vspace{0.28em}

$\bullet$ \emph{Bounded objects}: $\T^b\coloneqq \T^-\cap\T^+$;

\vspace{0.25em}

$\bullet$ \emph{Compact objects}: $\T^c\coloneqq\bigcup_{n\in\mathbb{N}}{\langle G\rangle}^{[-n,\,n]}$;

\vspace{-0.11em}
$\bullet$ \emph{Strongly bounded objects}: $\T^{sb}\coloneqq\bigcup_{n\in\mathbb{N}}\overline{\langle G\rangle}^{[-n,\,n]}$;

\vspace{0.28em}
$\bullet$ \emph{Bounded pseudo-compact objects}:
$$\T^b_c\coloneqq\T^b\cap(\bigcap_{n\in\mathbb{N}}{(\T^{c}\ast\T_G^{\leq{-n}})}).\vspace{-0.3em}$$

By definition, the \emph{extension} of the full subcategories $\mathcal{X}$ and $\mathcal{Y}$ of $\T$, denoted by $\mathcal{X} \ast \mathcal{Y}$, is the full subcategory of $\mathcal{T}$ consisting of objects $Z$ such that there is a (distinguished) triangle $X\rightarrow Z \rightarrow Y \rightarrow X[1]$ in $\mathcal{T}$ with $X \in \mathcal{X}$ and $Y \in \mathcal{Y}$. If $\mathcal{X} \ast \mathcal{X}\subseteq \mathcal{X}$, then $\mathcal{X}$ is said to be \emph{closed under extensions} in $\T$. Moreover, for integers $a\leq b$, the category ${\langle G \rangle}^{[a,\,b]}$ (respectively, $\overline{\langle G \rangle}^{[a,\,b]}$) denotes the smallest full subcategory of $\T$ containing $G[-n]$ for all integers $n\in [a,b]$ and closed under direct summands, finite direct sums (respectively, coproducts) and extensions.

Note that the inclusion $\T^c\subseteq \T^b$ holds if and only if $G$ is a bounded object, which is equivalent to the vanishing of $\Hom_{\mathcal{T}}(G, G[i])$ for $i\ll 0$. In our discussions, we always assume that $G$ is bounded, and therefore $\T^{sb}\subseteq \T^b$. This may explain why the objects of
$\T^{sb}$ are said to be \emph{strongly bounded}.

We now illustrate the above-mentioned six triangulated categories by an example. Let $\T:=\mathscr{D}(R\Modcat)$ be the unbounded derived category of a left coherent ring $R$. Then there are equalities (up to canonical equivalence):
$$\T^-=\mathscr{D}^{-}(R\Modcat),\; \T^+=\mathscr{D}^{+}(R\Modcat), \; \T^b=\mathscr{D}^b(R\Modcat),$$
$$\T^c=\mathscr{K}^b(\prj{R}), \;\T^{sb}=\mathscr{K}^b(R\text{-}{\rm Proj}),\; \T^b_c=\mathscr{D}^b(R\modcat),$$
where $R\Modcat$, $R\modcat$, $R\prj$ and $R\Prj$ denote the categories of (left) $R$-modules, finitely presented, finitely generated projective and projective $R$-modules, respectively.

As a vast generalization of derived module categories, Neeman introduced the notion of weakly approximable triangulated categories. This class of triangulated categories generalizes many important triangulated categories in practice, including the derived category of a non-positive differential graded ring, the derived category $\mathscr{D}_{\mathrm{qc}}(X)$ of unbounded complexes of $\mathscr{O}_X$-modules with quasi-coherent cohomology for a quasi-compact, quasi-separated scheme $X$, and the homotopy category of spectra. Weakly approximable triangulated categories have been a reasonable framework for addressing longstanding conjectures and providing new proofs and generalizations of some known theorems (for instance, see \cite{BCRPZ24,CHNS24,N21,N24A,N18,SZZ24}).

\vspace{0.15em}
Let us return to the abstract triangulated category $\mathcal{T}$. It is easy to see that every $t$-structure $(\T^{\leq 0},\T^{\geq 1})$ on $\mathcal{T}$ in the preferred equivalence class can be restricted to a bounded $t$-structure $(\T^{\leq 0}\cap\T^b,\T^{\geq 1}\cap\T^b)$ on the triangulated category $\T^b$, and all these bounded $t$-structures are equivalent. {However, $\mathcal{T}^c$ may not have any bounded $t$-structure.} A general discussion of the existence of a bounded $t$-structure on $\mathcal{T}^c$ (even on any {essentially  small} triangulated category) was done in \cite{BCRPZ24} by introducing the concept of finitistic dimension for triangulated categories. Following \cite[Definition 1.3]{BCRPZ24},  a triangulated category {$\mathcal{C}$} is said to have \emph{finite finitistic dimension} if there exists an object $U\in\mathcal{C}$ and an integer $n$ such that
$$\{X\in \mathcal{C}\mid \Hom_\mathcal{C}(U, X[i])=0, \; \forall\, i\leq -1\}\subseteq\bigcup_{n\le m}{\langle U\rangle}^{[n,\,m]}.\vspace{-0.5em}$$
This homological condition is not so strict since several classes of triangulated categories have been shown to have finite finitistic dimension; see \cite[Section 1.1 and Section 4]{BCRPZ24} for details.

The following result is a special case of \cite[Theorem 1.5]{BCRPZ24} which provides a categorical obstruction to the existence of bounded $t$-structures in terms of completions of triangulated categories.

\begin{thm}{\rm \cite[Corollary 1.7]{BCRPZ24}}\label{Compact}
Suppose that $\Hom_\mathcal{T}(G, G[i])=0$ for $i\gg 0$ and  the opposite category of $\mathcal{T}^{c}$ has finite finitistic dimension.

$(1)$ If $\mathcal{T}^{c}$ has a bounded $t$-structure, then $\mathcal{T}^{c}=\mathcal{T}_c^b$.

$(2)$ If $\mathcal{U}$ is a full triangulated subcategory of $\mathcal{T}$ with $\mathcal{T}^{c}\subseteq \mathcal{U}\subseteq \mathcal{T}_c^b$, then all bounded $t$-structures on $\mathcal{U}$ are equivalent.
\end{thm}

In this paper, we consider the question on the existence and uniqueness of bounded $t$-structures on the triangulated category $\T^{sb}$, an infinite-dimensional version of $\T^{c}$.  Our discussion depends on the following homological condition which plays a similar role in Theorem~\ref{Compact} as the finiteness of finitistic dimension.

\begin{defn}\label{defn-fd2}
We say that $\T$ has \emph{finite strong finitistic dimension} if there exists an integer $m$ such that\vspace{-0.4em}
$$\{X\in \T^{sb}\mid \Hom_\T(X, G[i])=0, \; \forall\, i\leq -1\}\subseteq\bigcup_{n\le m}\overline{\langle G\rangle}^{[n,\,m]},\vspace{-1em}$$
where $G$ is a compact generator of $\T$ and the category $\overline{\langle G \rangle}^{[n,\,m]}$ is a full subcategory of $\mathcal{T}$.
\end{defn}

Clearly, the finiteness of strong finitistic dimension is independent of the choice of compact generators. The terminology ``finite strong finitistic dimension" is motivated by the following proposition. For the precise definitions of weakly approximable (see \cite{N18}) and Gorenstein triangulated categories, we refer to Definitions~\ref{defn-app} and \ref{Gorenstein}, respectively. By Example \ref{Gexam}, the derived category of a non-positive Gorenstein differential graded algebra and $\mathscr{D}_{\mathrm{qc}}(X)$ for a proper Gorenstein scheme $X$ are weakly approximable and Gorenstein. Both categories have bounded compact generators.

\begin{prop}\label{FGD}
$(1)$ If $\T$ is a Gorenstein, weakly approximable triangulated category with a bounded compact generator, then it has finite strong finitistic dimension.

$(2)$ If $\T$ has finite strong finitistic dimension, then the opposite category of $\T^{\rm c}$ has finite finitistic dimension.
\end{prop}

Our main result of this paper reads as follows.

\begin{thm}\label{main1}
Let $\mathcal{T}$ be a weakly approximable triangulated category with a bounded compact generator. The following statements are true.

$(1)$ Up to equivalence, there exists a unique bounded $t$-structure on $\T^b$.

$(2)$ Suppose that $\T$ has finite strong finitistic dimension. Then any full triangulated subcategory $\U$ of $\T$ with $\T^{sb}\subseteq \U\subsetneq \mathcal{T}^b$ admits no bounded $t$-structure. Thus $\Tsb$ admits a bounded $t$-structure if and only if $\Tsb = \T^{b}$.
\end{thm}

Recall from \cite[Lemma 5.2(2)]{CCZ26} that, for a weakly approximable triangulated category $\T$ with a bounded compact generator, the equality $\Tsb = \T^{b}$ holds if and only if $\T$ has \emph{finite global dimension}, that is, there is a natural number $d$ such that $\Hom_{\mathcal{T}}(\mathcal{T}_G^{\geq 0}\cap\mathcal{T}^{b},\mathcal{T}_G^{\leq-(d+1)}\cap\mathcal{T}^{b})=0$. For the case $\T=\mathscr{D}(S\Modcat)$ with an ordinary ring $S$, we see from \cite[Remark 5.3]{CCZ26} that $\T$ has finite global dimension if and only if $S$ has finite global dimension. For the case $\T=\mathscr{D}_{\mathrm{qc}}(X)$ with a finite-dimensional noetherian scheme $X$, we see from Lemma \ref{lasttech} that $\T$ has finite global dimension if and only if $X$ is regular.
Thus a combination of Theorem \ref{main1}(2) with Proposition \ref{FGD}(1) implies the following result.

\begin{cor}\label{mainapp}
$(1)$ Let $\Lambda$ be a Gorenstein Artin algebra. Then $\Kb{\Pmodcat{\Lambda}}$ admits a bounded $t$-structure if and only if $\Lambda$ has finite global dimension.

$(2)$ Let $X$ be a finite-dimensional noetherian scheme with a proper, Gorenstein morphism $X\to\Spec(R)$ for a commutative Gorenstein Artin ring $R$. Then $\mathscr{D}_{\mathrm{qc}}(X)^{sb}$ admits a
bounded $t$-structure if and only if $X$ is regular.
\end{cor}

By using recollements of triangulated categories, we can construct more weakly approximable triangulated categories with finite strong finitistic dimension (Proposition \ref{construct1}), and these categories may not be Gorenstein by Remark \ref{Non-Gorenstein}. Thus Theorem \ref{main1} can also be applied to some non-Gorenstein algebras glued by Gorenstein algebras via stratifying ideals (Corollary \ref{app2}).

\emph{Structure of the paper:} In Section \ref{2}, we fix some notation and recall some basic definitions including $t$-structures, weakly approximable triangulated categories and approximating systems. In Section \ref{3}, we prove Theorem \ref{main1}. As a crucial technique, a method for lifting bounded above $t$-structures from $\U$ (any intermediate triangulated category between $\T^{sb}$ and $\T^b$) to $\T^b$ is given in Lemma \ref{lem-lifting-$t$-stucture}. In Section \ref{4}, we first introduce the concept of Gorenstein triangulated categories (Definition \ref{Gorenstein}) together with algebraic and geometric examples, and prove Proposition \ref{FGD} and Corollary \ref{mainapp}. Then we construct triangulated categories with finite strong finitistic dimension via recollements of triangulated categories (Proposition~\ref{construct1}). Finally, we apply Theorem \ref{main1} to recollements of derived module categories and obtain Corollaries \ref{construct3} and \ref{app2}. This links the existence of bounded $t$-structures to the regularity of algebras.

\section{Preliminaries}\label{2}

In this section, we briefly fix notation and recall some definitions and basic facts used in the paper.

\subsection{(Bounded) $t$-structures and weakly approximable triangulated categories}\label{BWAT}

Let $\T$ be a triangulated category with the shift functor $[1]$. For a full subcategory $\mathcal{X}$ of $\T$, we denote the full subcategories of $\mathcal{T}$ right and left orthogonal to $\mathcal{X}$ separately by
$$\mathcal{X}^{\perp}\coloneqq\{M\in \T\mid \Hom_{\T}(X,M)=0, \;\forall\,X\in \mathcal{X}\} \;\; \mbox{and}\;\;{}^{\perp}\mathcal{X}\coloneqq\{M\in \T\mid \Hom_{\T}(M, X)=0,\;\forall\,X\in \mathcal{X}\}.$$
Suppose that $\T$ admits (small) coproducts. An object $X \in \T$ is said to be \emph{compact} if the Hom-functor $\Hom_{\T}(X,-) \colon \T \to\mathbb{Z}\Modcat$ preserves coproducts. The full subcategory of $\T$ consisting of all compact objects is denoted by $\T^{c}$. We say that $\T$ is \emph{compactly generated} if there exists a set $\mathcal{G}\subseteq \T^{c}$ such that $\bigcap_{i\in\mathbb{Z}}(\mathcal{G}[i])^{\perp}=\{0\}$ (or equivalently, $\T$ is exactly the smallest full triangulated subcategory of $\T$ containing $\mathcal{G}$ and closed under coproducts). In this case, $\T^c$ is the smallest full triangulated subcategory of $\T$ containing $\mathcal{G}$ and closed under direct summands. If further $\mathcal{G}=\{G\}$, then $\T$ is called a \emph{monogenic} triangulated category.

We first fix some notation in the following definition.

\begin{defn}{\rm \cite[Reminder 1.12]{N18}}\label{notation} Let $\T$ be a triangulated category with coproducts and with $G \in \T$.

{\rm (1)} For integers $a \leq b$, let
$G[a,b] \coloneqq \{ G[-i] \mid i \in \mathbb{Z},\ a \leq i \leq b \}$.
For later use, we extend this notation to allow $a$ or $b$ to be infinite; for example,
$G(-\infty, b] \coloneqq \{ G[-i] \mid i \in \mathbb{Z},\ i \leq b \}$.

\vspace{0.3em}

{\rm (2)} Let $a\leq b$ be integers (possibly infinite). We denote by $\langle G \rangle^{[a,\,b]}$ the smallest full subcategory of $\T$ which contains $G[a,b]$ and is closed under direct summands, finite direct sums and extensions. Clearly, $\langle G\rangle^{(-\infty,+\infty)}=\bigcup_{i\geq 0}\langle G\rangle^{[-i,\,i]}$. For simplicity, we write \(\langle G \rangle\) for \(\langle G \rangle^{(-\infty,+\infty)}\).

{\rm (3)} Let $a\leq b$ be integers (possibly infinite). We denote by $\overline{\langle G \rangle}^{[a,\,b]}$ the smallest full subcategory of $\T$ which contains $G[a,b]$ and is closed under direct summands, coproducts and extensions.

{\rm (4)} The object $G$ is called a \emph{compact generator} of $\T$ if $G\in \T^{c}$ and $\T=\overline{\langle G \rangle}^{(-\infty, +\infty)}$. If, in addition, $\Hom_{\T}(G,G[i])=0$ for $i\ll 0$, then $G$ is called a \emph{bounded compact generator} of $\T$.
\end{defn}

The definition of $t$-structures on a triangulated category is standard.

\begin{defn}{\rm \cite[Definition 1.3.1]{BBD}}
\label{defn-$t$-structure}
Let $\mathcal{T}$ be a triangulated category. A pair of full subcategories $(\mathcal{T}^{\leq 0},\mathcal{T}^{\geq 1})$ in $\mathcal{T}$ is called a \emph{$t$-structure} on $\mathcal{T}$ if the following conditions are satisfied:

{\rm (T1)} $\mathcal{T}^{\leq 0}[1] \subseteq \mathcal{T}^{\leq 0}$ and $\mathcal{T}^{\geq 1} \subseteq\mathcal{T}^{\geq 1}[1]$;

{\rm (T2)} $\Hom_{\mathcal{T}}(\mathcal{T}^{\leq 0},\mathcal{T}^{\geq 1})=0$;

{\rm (T3)} For any object $X \in \mathcal{T}$, there is a triangle
$X^{\leq 0} \rightarrow X \rightarrow X^{\geq 1} \rightarrow X^{\leq 0}[1]$ in $\T$
with $X^{\leq 0} \in \mathcal{T}^{\leq 0}$ and $X^{\geq 1} \in \mathcal{T}^{\geq 1}$.
\end{defn}
Let $(\T^{\leq 0},\T^{\geq 1})$ be a $t$-structure on $\T$.
For each $n \in \mathbb{Z}$, we set $\T^{\leq n}\coloneqq \T^{\leq 0}[-n]$ and $\T^{\geq n+1}\coloneqq \T^{\geq 1}[-n]$. Then $(\T^{\leq n})^{\perp}=\T^{\geq n+1}$ and
$^{\perp}(\T^{\geq n})=\T^{\leq n-1}$.
Thus $\T^{\leq n}$ and $\T^{\geq n+1}$ determine each other and $(\T^{\leq n},\T^{\geq n+1})$ is also a $t$-structure on $\T$.

The $t$-structure $(\mathcal{T}^{\leq 0},\mathcal{T}^{\geq 1})$ is said to be \emph{bounded above} (resp.\ \emph{bounded below}) if $\T = \bigcup_{m=1}^{\infty} \T^{\leq m}$ (resp.\ $\T = \bigcup_{m=1}^{\infty} \T^{\geq -m}$). It is \emph{bounded} if it is both bounded above and bounded below; \emph{nondegenerate} if $\bigcap_{i\in\mathbb{Z}}\T^{\leq i}=\bigcap_{i\in\mathbb{Z}}\T^{\geq i}=\{0\}.$ Note that if $(\mathcal{T}^{\leq 0},\mathcal{T}^{\geq 1})$ is bounded, then it is nondegenerate.

Let $\mathcal{H}\coloneqq\T^{\leq 0}\cap\T^{\geq 0}$. Then $\mathcal{H}$ is an abelian category and is usually called the \emph{heart} of the $t$-structure $(\T^{\leq 0},\T^{\geq 1})$. Let $H:\T\to\mathcal{H}$ be the associated cohomological functor (see \cite[Theorem 1.3.6]{BBD}). Further, we define $H^i:=H\circ[i]$ for $i\neq 0$, and $H^0\coloneqq H$.

Two $t$-structures $(\T^{\leq 0}_i,\T^{\geq 1}_i)$ for $i=1,2$ on $\T$ are said to be \emph{equivalent} if there exists some natural number $n$ such that $\T^{\leq -n}_1\subseteq \T^{\leq 0}_2\subseteq \T^{\leq n}_1$, or equivalently,  $\T^{\geq n}_1\subseteq \T^{\geq 0}_2\subseteq \T^{\geq -n}_1$. This defines an equivalence relation on the class of all $t$-structures on $\T$.

The following lemma gives a canonical construction of $t$-structures 
generated by compact objects. The result was originally established in 
\cite[Theorem A.1]{TLS03}. Alternative proofs can be found in 
\cite[Theorem 12.1]{KN13} and \cite[Theorem 3.0.1]{CHNS24}.

\begin{lem}{\rm \cite[Theorem A.1]{TLS03}}\label{lem-p-w-gen}
Let $\T$ be a triangulated category with coproducts and let $G$ be a compact object in $\T$. Then there is a $t$-structure on $\T$ generated by $G$, that is,
$$
(\T_{G}^{\leq 0},\T_{G}^{\geq 1})\coloneqq\big(\,\overline{\langle G\rangle}^{(-\infty,0]},G(-\infty,0]^{\perp}\big).
$$
Moreover, the heart $\mathcal{H}$ of $(\T_{G}^{\leq 0},\T_{G}^{\geq 1})$ is closed under coproducts in $\T$ and the cohomological functor $H\colon\T\to\mathcal{H}$ respects coproducts in $\T$.
\end{lem}

Following \cite[Definition~1.18]{N18}, for a compactly generated triangulated category $\T$ with a single compact generator $G$, the equivalence class of the $t$-structure~$(\T_{G}^{\leq 0}, \T_{G}^{\geq 1})$ in Lemma~\ref{lem-p-w-gen} is called the \emph{preferred equivalence class}.

The following lemma will be used later.

\begin{lem}{\rm \cite[Proposition 1.3.7]{BBD}}\label{tec1}
Let $(\mathcal{T}^{\leq 0},\mathcal{T}^{\geq 1})$ be a nondegenerate $t$-structure on a triangulated category $\mathcal{T}$ and let $n\in \mathbb{Z}$.
Then $X \in \mathcal{T}^{\leq n}$ (resp.\ $X \in \mathcal{T}^{\geq n}$) if and only if $H^{i}(X)=0$ for all $i\geq n+1$ (resp.\ $i\leq n-1$). In particular, a morphism $f \colon X \to Y$ in $\mathcal{T}$ is an isomorphism if and only if $H^{i}(f): H^i(X)\to H^i(Y)$ are isomorphisms for all $i\in\mathbb{Z}$.
\end{lem}

Now, we recall the definition of homotopy colimits in triangulated categories.

\begin{defn}{\rm \cite[Definition 1.6.4]{N01}}\label{HCT}
Let $\T$ be a triangulated category such that countable coproducts exist in $\T$.  Let
$$X_1\lraf{f_1}X_2\lraf{f_2}X_3\lraf{f_3}\cdots\lraf{f_n}X_{n+1}\lra\cdots$$
be a sequence of objects and morphisms in $\T$. The \emph{homotopy colimit} of this sequence, denoted by $\underrightarrow{\hocolim}(X_*)$,
is given, up to non-canonical isomorphism, by the triangle
$$
\bigoplus_{n\geq 1}X_n\lraf{(1-f_*)} \bigoplus_{n\geq 1}X_n\lra \underrightarrow{\hocolim}(X_*)\lra \bigoplus_{n\geq 1}X_n[1]
$$
where the morphism $(1-f_*)$  is induced by  $\left(\begin{smallmatrix} \mathrm{Id}_{X_i} \\ -f_i \end{smallmatrix}\right): X_i\to X_i\oplus X_{i+1}\subseteq \bigoplus_{n\geq 1}X_n$ \,for all $i\geq 1$.
\end{defn}

The following lemma is well known to experts. For the reader's convenience we include a proof.

\begin{lem}\label{tec2}
Let $\T$ be a triangulated category with coproducts and let $(\T^{\leq 0}, \T^{\geq 1})$ be a $t$-structure generated by a compact object in $\T$. Consider the following sequence in $\T$:
\[
(X_{*},f_{*}):\; X_{1}\xrightarrow{f_{1}}X_{2}\xrightarrow{f_{2}}X_{3}\xrightarrow{f_{3}}\cdots\rightarrow  X_n\xrightarrow{f_n}X_{n+1}\rightarrow\cdots.
\] Suppose that, for each integer $i$, there exists a positive integer $n_i$ such that $H^i(f_j):H^i(X_j)\to H^i(X_{j+1})$ is an isomorphism in $\T$ for $j\geq n_i$. Then for each $i\in\mathbb{Z}$, there is a canonical isomorphism in the heart of the $t$-structure:
$$\colims\,(H^i(X_{*}))\simeq H^i(\Hocolim(X_{*})).$$
\end{lem}
\begin{proof}
By Definition \ref{HCT}, there is a triangle in $\T$:
\[
\begin{tikzcd}
\bigoplus_{n\geq 1}X_n \arrow[r, "\sigma"] & \bigoplus_{n\geq 1}X_n \arrow[r] & \Hocolim(X_{*}) \arrow[r] & {(\bigoplus_{n\geq 1}X_n)[1]}
\end{tikzcd}
\]
with $\sigma:=(1-f_*)$. Applying $H^i(-)$ to the triangle yields a long exact sequence in the heart $\mathcal{H}$ of $(\T^{\leq 0}, \T^{\geq 1})$:
{\small\[
\begin{tikzcd}
H^{i}(\bigoplus_{n\geq 1}X_n) \arrow[r, "H^{i}(\sigma)"] & H^{i}(\bigoplus_{n\geq 1}X_n) \arrow[r] & H^{i}(\Hocolim(X_{*})) \arrow[r] & H^{i+1}(\bigoplus_{n\geq 1}X_n) \arrow[r, "H^{i+1}(\sigma)"] & H^{i+1}(\bigoplus_{n\geq 1}X_n).
\end{tikzcd}
\]}
Since $H(-)$ respects coproducts in $\T$ by Lemma \ref{lem-p-w-gen}, there is a canonical isomorphism $$(\ast)\quad \bigoplus_{n\geq 1}H^{i}(X_n)\lra H^{i}(\bigoplus_{n\geq 1}X_n)$$ for each $i\in\mathbb{Z}$.
This implies $\Coker(H^{i}(\sigma))\simeq\colims(H^{i}(X_{*}))$. To show $\colims\,(H^i(X_{*}))\simeq H^i(\Hocolim(X_{*}))$, we only need to prove that $H^{i+1}(\sigma)$ is a monomorphism in $\mathcal{H}$.

In fact, by $(\ast)$, there exists a triangle in $\T$:
\[
\begin{tikzcd}
{\Hocolim (H^{i+1}(X_{*}))[-1]} \arrow[r] & H^{i+1}(\bigoplus_{n\geq 1}X_n) \arrow[r, "H^{i+1}(\sigma)"] & H^{i+1}(\bigoplus_{n\geq 1}X_n) \arrow[r] & \Hocolim (H^{i+1}(X_{*})).
\end{tikzcd}
\]
Since $H^{i+1}(f_j)$ is an isomorphism in $\T$ for all $j\geq n_{i+1}$, we have
$$\Hocolim(H^{i+1}(X_{*}))\simeq H^{i+1}(X_{n_{i+1}})\in\mathcal{H}\subseteq \T^{\geq 0}.$$
Note that $\Ker(H^{i+1}(\sigma))\simeq(\Hocolim(H^{i+1}(X_{*}))[-1])^{\leq 0}$ by the abelian structure of $\mathcal{H}$. It follows that $\Ker(H^{i+1}(\sigma))=0$, and thus $H^{i+1}(\sigma)$ is a monomorphism.
\end{proof}

We are interested in a special class of monogenic triangulated categories introduced by Neeman.

\begin{defn}{\rm \cite[Definition 1.25]{N18}}\label{defn-app}
A triangulated category $\T$ with coproducts is said to be \emph{weakly approximable} if
it has a compact generator $G$ and a $t$-structure $(\mathcal{T}^{\le 0},\mathcal{T}^{\ge 1})$
together with an integer $A>0$ such that

{\rm (1)}
$G[A]\in \mathcal{T}^{\leq0}$ and $\Hom_{\mathcal{T}}(G[-A],\mathcal{T}^{\leq 0})=0$.

{\rm (2)}
Each object $F\in\mathcal{T}^{\leq 0}$
admits a triangle $E\ra F\ra D \ra E[1]$ with $E\in  \overline{\langle G\rangle}^{[-A,\,A]}$ and
$D\in\mathcal{T}^{\le -1}$.
\end{defn}

By \cite[Facts~1.26]{N18}, the notion of weakly approximable triangulated categories does not depend on the choice of compact generators and $t$-structures in the preferred equivalence class. In particular, the $t$-structure $(\T^{\le 0},\T^{\ge 1})$ in Definition \ref{defn-app} lies in the preferred equivalence class. So, when addressing a weakly approximable triangulated category $\T$ with a compact generator $G$, we can always take the $t$-structure $(\T^{\leq 0},\T^{\geq 1})$ in Definition \ref{defn-app} to be the canonical $t$-structure $(\T_G^{\leq 0},\T_G^{\geq 1})$ associated with $G$ (see Lemma~\ref{lem-p-w-gen}). Moreover, $\Hom_{\T}(G,G[i])=0$ for $i\gg0$ by Definition~\ref{defn-app}(1), and every $t$-structure in the preferred equivalence class is nondegenerate by \cite[Lemma~3.6]{BNP23}.

We end this subsection with the following key result.

\begin{lem}{\rm\cite[Corollary 3.2]{N18}}\label{lem-app-induction}
Let $\mathcal{T}$ be a compactly generated triangulated category with a compact generator $G$, and let $(\T^{\le 0},\T^{\ge 1})$ be a $t$-structure in the preferred equivalence class.

If there exists an integer $A>0$ satisfying {\rm Definition~\ref{defn-app}(2)}, then for every integer $m>0$ and every object $F\in \mathcal{T}^{\leq 0}$, there exists a triangle $C_m \to F \to D_m \to C_m[1]$ in $\T$ such that $C_m\in \overline{\langle G\rangle}^{[1-m-A,\,A]}$ and $D_m\in \mathcal{T}^{\le -m}$.
\end{lem}

\subsection{Approximating systems}

As a preparation for showing Theorem \ref{main1}, we introduce the following definition that is a modification of \cite[Definition 9.3]{N18}.

\begin{defn}\label{strong approximating system}
Let $\T$ be a triangulated category equipped with a $t$-structure $(\T^{\le 0},\T^{\ge 1})$,
$\mathcal{A}$ a full subcategory of $\mathcal{T}$ and
$(E_{*},f_{*}) : E_{1}\xrightarrow{f_{1}}E_{2}\xrightarrow{f_{2}}E_{3}\xrightarrow{f_{3}}\cdots\to E_m\xrightarrow{f_m}E_{m+1}\rightarrow\cdots$
a sequence of objects and morphisms in $\mathcal{A}$.

$(1)$ The sequence $ (E_{*},f_{*}) $ is called a \emph{strong $\mathcal{A}$-approximating system} if for each $ m \in \mathbb{N}^+ $, the morphism $ H^i(f_{m}) $ is an isomorphism for all integers $ i \geq -m $.

$(2)$ Let $ F \in \mathcal{T} $. If there exist morphisms $ g_j: E_{j} \to F $ for all $ j \in \mathbb{N}^+ $ such that $ g_j = g_{j+1}f_{j} $ for all $ j \in \mathbb{N}^+ $, and for each $ m \in \mathbb{N}^+ $, the morphism $ H^i(g_{m}) $ is an isomorphism for all integers $ i \geq -m $, then $ (E_{*},f_{*}) $ is called a \emph{strong $\mathcal{A}$-approximating system} for $ F $.
\end{defn}

In the rest of this subsection, let $\T$ be a weakly approximable triangulated category with a compact generator $G$ and with a $t$-structure $(\T^{\le 0},\T^{\ge 1})$ in the preferred equivalence class. We now take a sequence $(E_*,f_*): E_{1}\xrightarrow{f_{1}}E_{2}\xrightarrow{f_{2}}E_{3}\xrightarrow{f_{3}}\cdots$ in the full subcategory $\Tsb$ of $\T$ consisting of strongly bounded objects.

The following lemma is inspired by \cite[Lemma~9.5]{N18}.

\begin{lem}\label{lem-strong approximating system}
{\rm (1)} If $(E_{*},f_{*})$ is a strong $\Tsb$-approximating system, then it is a strong $\Tsb$-approximating system for $\Hocolim(E_{*})$. Moreover, $\Hocolim(E_{*})$ belongs to $\T^{-}$.

{\rm (2)} Let $F \in \T^{-}$ and let $(E_{*}, f_{*})$ be a strong $\Tsb$-approximating system for $F$. Then the (non-canonical) morphism $\Hocolim(E_{*})\to F$ is an isomorphism.

{\rm (3)} Any object $F\in\T^{-}$ has a strong $\Tsb$-approximating system.
\end{lem}
\begin{proof}
By Definition \ref{HCT}, there is a triangle in $\T$:
$$
\bigoplus_{k\geq 1}E_k\lraf{(1-f_*)} \bigoplus_{k\geq 1}E_k\lraf{g} \underrightarrow{\hocolim}(E_*)\lra \bigoplus_{k\geq 1}E_k[1].
$$
Let $E:=\Hocolim(E_{*})$. For any $j>0$, we denote by $e_{j}:E_{j}\to\bigoplus_{k\geq 1}E_{k}$ the canonical injection and define $g_{j}=ge_{j}: E_j\to E$. Clearly, $g_{j}=g_{j+1}f_{j}$.

(1) Suppose that $(E_{*},f_{*})$ is a strong $\Tsb$-approximating system. Then for each $i\in\mathbb{Z}$,
the morphisms $H^i(f_m)$ are isomorphisms for all $m\in \mathbb{N}^+ $ with $m\geq -i$. It follows from Lemma~\ref{tec2} that there is a canonical isomorphism $\colims\,(H^i(E_{*}))\simeq H^i(E)$. Note that $\colims\,(H^i(E_{*}))\simeq H^i(E_m)$ for any $m\geq -i$. Thus $H^i(g_m) $ are isomorphisms whenever $ i \geq -m $. It follows that $H^i(E)\simeq H^i(E_1)$ for any $i\geq -1$. Moreover, since $E_{1} \in \Tsb \subseteq \T^{-}$, it belongs to $\T^{\leq n}$ for some $n > 0$. By Lemma \ref{tec1}, $H^{i}(E_{1}) = 0$  for $i > n$, which forces $H^i(E)=0$ for $i>n$. Still by Lemma \ref{tec1}, we have  $E\in \T^{\leq n}\subseteq \T^{-}$.

(2) Clearly, there exists a morphism $ f\colon E\to F$ such that, for every $j \geq 1 $, the following diagram is commutative:
\[
\begin{tikzcd}
E_j \arrow[r, "g_j"] \arrow[rd] & E \arrow[d, "f"] \\
& F.
\end{tikzcd}
\]
Note that, for every $i\in\mathbb{Z}$, the morphisms $H^{i}(E_{n_i})\to H^{i}(F)$ and $H^{i}(E_{n_i})\to H^{i}(E)$ are isomorphisms, where $n_i:=\max\{-i,1\}$. It follows that $H^{i}(f)$ is an isomorphism. Now,  by Lemma \ref{tec1}, $f$ is an isomorphism.

(3) Without loss of generality, assume $F\in\T^{\leq 0}$. We wish to construct a strong $\Tsb$-approximating system $(C_{*}, h_{*}): C_{1}\xrightarrow{h_{1}}C_{2}\xrightarrow{h_{2}}C_{3}\xrightarrow{h_{3}}\cdots$ for $F$.

By Lemma~\ref{lem-app-induction}, there exists a triangle $C_{1}\xrightarrow{g_{1}} F\to D_{1}\to C_{1}[1]$ in $\T$ with $C_{1}\in\overline{\langle G\rangle}^{[-2-A,\,A]}$ and $D_{1}\in\T^{\leq -3}$. Then $H^i(D_1)=0=H^i(D_1[-1])=0$ for any $i \geq -1$. It follows that $C_{1}\in\Tsb$ and $H^{i}(g_{1})$ is an isomorphism for $i \geq -1$. This completes the construction of $C_1$ in the $\Tsb$-approximating system. Further, we claim that $C_{1}\in{}^{\perp}(\T^{\leq -2-2A})$. In fact, since $\Hom_{\T}(G[-A],\T^{\leq 0})=0$ by Definition \ref{defn-app}, we have $G[-2-A,A]\subseteq{}^{\perp}(\T^{\leq -2-2A})$. Note that ${}^{\perp}(\T^{\leq -2-2A})\subseteq \T$ is closed under direct summands, coproducts and extensions. Thus $\overline{\langle G\rangle}^{[-2-A,\,A]}\subseteq^{\perp}(\T^{\leq -2-2A})$, which implies $C_{1}\in{}^{\perp}(\T^{\leq -2-2A})$.

In the following, we apply Lemma~\ref{lem-app-induction} again to $F$ and obtain a triangle  $C_{2}\xrightarrow{g_{2}}F\to D_{2}\to C_{2}[1]$ in $\T$ with $C_{2}\in\overline{\langle G\rangle}^{[-1-3A,\,A]}$ and $D_{2}\in\T^{\leq -2-2A}$. Then $C_{2}\in\Tsb$. By $-2-2A \leq -4$, the morphism $H^{i}(g_{2})$ is an isomorphism for all $i \geq -2$. Moreover, since $\Hom_{\T}(C_{1}, D_{2}) = 0$, there exists a morphism $h_{1}\colon C_{1} \to C_{2}$ satisfying $g_{1} = g_{2}h_{1}$. This completes the construction of $C_2$ together with the connecting morphism $h_1$ in the $\Tsb$-approximating system. Now we can use Lemma~\ref{lem-app-induction} repeatedly and construct $(C_{*}, h_{*})$ by induction.
\end{proof}

Combining Lemma~\ref{tec1} and Lemma~\ref{lem-strong approximating system}(1), we obtain the following result.

\begin{cor}\label{lem-cauchy seq}
Let $(E_{*},f_{*})$ be a strong $\Tsb$-approximating system.

{\rm (1)} For each $m\in\mathbb{N}^{+}$, there is a triangle $K_{m}\lra E_{m}\lraf{f_{m}}E_{m+1}\lra K_{m}[1]$ with $K_{m}\in\T^{\leq -m}$.

{\rm (2)} For each $m\in\mathbb{N}^{+}$, there is a triangle $L_{m}\lra E_{m}\lraf{g_{m}}\Hocolim(E_{*})\lra L_{m}[1]$ with $L_{m}\in\T^{\leq -m}$.
\end{cor}

\section{Equivalence and existence of bounded $t$-structures}\label{3}

In this section, we prove Theorem~\ref{main1}. As preparation, we recall the following definition which will also be used in the definition of Gorenstein triangulated categories in the next section.

\begin{defn}{\rm \cite[Remark 8.5.22]{N01}}\label{def bcdual}
Let $\T$ be an $R$-linear, compactly generated triangulated category over a commutative ring $R$, $I$ an injective cogenerator in $R\Modcat$, and $G$ a compact object of $\T$. An object $E\in\T$ is called the \emph{Brown-Comenetz dual} of $G$ (with respect to $I$) if there exists a natural isomorphism of functors from $\T\opp$ to $R\Modcat$:
\[
\Hom_{\T}(-,E) \simeq \Hom_{R}(\Hom_{\T}(G,-),I).
\]
\end{defn}
The classical Brown representability theorem (for example, see \cite[Theorem~3.1]{N96}) implies that for every $G\in\T^c$, the functor $\Hom_{R}(\Hom_{\T}(G,-),I)$ is representable. Thus the Brown-Comenetz dual of any compact object exists and is unique up to isomorphism by the Yoneda lemma. Note that the notion of Brown-Comenetz duality originates from \cite{BZ76}.

\medskip
{\bf Proof of Theorem \ref{main1}(1).}
Let $G$ be a compact generator of $\T$. Then $(\T_G^{\leq 0},\T_G^{\geq 1})$ is a $t$-structure on $\T$ by Lemma \ref{lem-p-w-gen}. Restricting this $t$-structure to $\T^b$ yields a bounded $t$-structure $(\T_G^{\leq 0}\cap\T^b,\T_G^{\geq 1}\cap\T^b)$ on $\T^b$. Let $(\mathcal{D}^{\leq 0},\mathcal{D}^{\geq 1})$ be another bounded $t$-structure on $\T^b$. To show the equivalence of these two $t$-structures, it suffices to show that there exist integers $t$ and $s$ such that
$$\mathcal{D}^{\leq t}\subseteq\T_{G}^{\leq 0}\cap\T^b\subseteq\mathcal{D}^{\leq s}.$$

In fact, since $\T$ is $\mathbb{Z}$-linear and compactly generated, the Brown-Comenetz dual $E$ of $G$ exists, that is, $E\in\T$ and there is a natural isomorphism $$(\dag):\quad \Hom_{\T}(-,E) \simeq D\Hom_{\T}(G,-):\;\T\opp\lra \mathbb{Z}\Modcat,$$
where $D:=\Hom_{\mathbb{Z}}(-,\mathbb{Q}/\mathbb{Z})$.
Note that $\Hom_{\T}(G,G[i])=0$ for $i\gg 0$ by the weak approximation of $\T$, and that $\Hom_{\T}(G,G[i])=0$ for $i\ll 0$ by the boundedness of $G$. It follows from $(\dag)$ that $\Hom_{\T}(G[i],E)\simeq D\Hom_{\T}(G,G[i])=0$ for $|i|\gg 0$. Further, since $\T$ is weakly approximable,
we see from \cite[Corollary 3.10]{BNP23} that $\mathcal{T}^b=\{X\in \mathcal{T}\mid \Hom_{\mathcal{T}}(G[i],X)=0 \;\text{for}\; |i|\gg 0\}$. Thus $E\in\T^b$.

Since $(\mathcal{D}^{\leq 0},\mathcal{D}^{\geq 1})$ is a bounded $t$-structure on $\T^b$, we have $E\in\mathcal{D}^{\geq m+1}$ for some integer $m$. Now, let $X\in\mathcal{D}^{\leq m}$. Then $\Hom_{\T}(X,E[i])=0$ for any $i\leq 0$. Still by $(\dag)$, there are isomorphisms $\Hom_{\T}(X,E[i])\simeq D\Hom_{\T}(G[i],X)$ for $i\in\mathbb{Z}$, and therefore $D\Hom_{\T}(G[i],X)=0$ for $i\leq 0$. Since $\mathbb{Q}/\mathbb{Z}\in\mathbb{Z}\Modcat$ is a cogenerator, the functor $D$ is faithful. Consequently, $\Hom_{\T}(G[i],X)=0$ for $i\leq 0$, that is, $X\in G[0, +\infty)^{\bot}$. Note that $G[-2A,+\infty)^{\perp}\subseteq \mathcal{T}^{\leq -A-1}$ by \cite[Lemma 3.9(iv)]{BNP23}, where $A$ is the chosen integer in Definition~\ref{defn-app}. Thus $X\in \T_G^{\leq A-1}$. This shows $\mathcal{D}^{\leq m}\subseteq\T_{G}^{\leq A-1}\cap\T^b$. Let $t:=m-A+1$. Then $\mathcal{D}^{\leq t}\subseteq\T_{G}^{\leq 0}\cap\T^b$.

Since $G$ is compact and bounded, we have $G\in\T^c\subseteq\T^b$. Then $G\in\mathcal{D}^{\leq s}$ for some integer $s$. Now, we consider the orthogonal subcategory $\mathcal{W}:={}^{\perp}[(\mathcal{D}^{\leq s})^{\perp}]$ of $\T$; see Section \ref{BWAT} for the definitions of left and right orthogonal subcategories of $\T$. Then $\mathcal{D}^{\leq s}\subseteq\mathcal{W}\cap\T^b$. Since $(\mathcal{D}^{\leq 0},\mathcal{D}^{\geq 1})$ is a $t$-structure on $\T^b$, it is clear that $\mathcal{D}^{\geq s+1}=(\mathcal{D}^{\leq s})^{\perp}\cap\T^b\subseteq (\mathcal{D}^{\leq s})^{\perp}$ and $\mathcal{D}^{\leq s}= {}^{\perp}(\mathcal{D}^{\leq s+1})\cap\T^b$. This implies $\mathcal{W}\cap\T^b\subseteq\mathcal{D}^{\leq s}$, and thus $\mathcal{W}\cap\T^b=\mathcal{D}^{\leq s}$. Note that $\mathcal{W}$ is closed under coproducts and extensions in $\T$. It is also closed under positive shifts in $\T$ because $\mathcal{D}^{\leq s}$ is closed under positive shifts in $\T^b$. Since $\T_{G}^{\leq 0}=\overline{\langle G\rangle}^{(-\infty,0]}$ by Lemma \ref{lem-p-w-gen}, we have $\T_{G}^{\leq 0}\subseteq\mathcal{W}$,  and further
$\T_{G}^{\leq 0}\cap\T^b\subseteq\mathcal{W}\cap\T^b=\mathcal{D}^{\leq s}$. \hfill$\square$

\medskip
\begin{rem}
In \cite[Lemma~4.4]{XZ26}, it was shown that all bounded $t$-structures on the bounded derived category of a Grothendieck category are equivalent.
\end{rem}

From now on, let $\T$ be a weakly approximable triangulated category with a bounded compact generator $G$ and with a $t$-structure $(\T^{\leq 0},\T^{\geq 1})$ in the preferred equivalence class. Without loss of generality, assume $(\T^{\leq 0},\T^{\geq 1}) = (\T_{G}^{\leq 0},\T_{G}^{\geq 1})$. Then
$$\T^c=\langle G \rangle \subseteq \T^{sb}\subseteq\T^b$$ (for example, see \cite[Proposition~3.4]{CCZ26}). Let $\U$ be any full triangulated subcategory of $\T$ satisfying $\Tsb\subseteq\U\subseteq\T^b$.
Then $\T^c\subseteq \mathcal{U}$.

We first establish the following simple result.
\begin{lem}\label{lem-bounded-above}
Let $(\mathcal{U}^{\le 0},\mathcal{U}^{\ge 1})$ be a $t$-structure on $\U$. The following statements are equivalent.

{\rm (1)} $\T^{\leq -k}\subseteq{}^{\perp}(\U^{\geq 1})$ for some $k>0$.

{\rm (2)} $\T^{\leq -k} \cap \U \subseteq \mathcal{U}^{\leq 0}$ for some $k > 0$.

{\rm (3) }$(\mathcal{U}^{\le 0},\mathcal{U}^{\ge 1})$ is a bounded above $t$-structure on $\mathcal{U}$.
\end{lem}
\begin{proof}
Since $(\mathcal{U}^{\le 0},\mathcal{U}^{\ge 1})$ is a $t$-structure on $\U$, we have $\U \cap{}^{\perp}(\U^{\geq 1})=\mathcal{U}^{\leq 0}$. Thus $(1)$ implies $(2)$.

Assume that $(2)$ holds. Then $\T^{\leq r} \cap \U \subseteq \mathcal{U}^{\leq r+k}$ for each $r\in \mathbb{Z}$. Since $\U\subseteq\T^{b}$, we have $\T^{\leq r} \cap \U=\T^{\leq r}\cap\T^{b}\cap \U$. Recall that $(\T^{\leq 0}\cap\T^{b},\T^{\geq 1}\cap\T^{b})$ is a bounded $t$-structure on $\T^{b}$ by the proof of Theorem \ref{main1}(1). Consequently, there are equalities
\begin{align*}
\U&=\T^{b}\cap\U
=(\bigcup_{r\in\mathbb{Z}}(\T^{\leq r}\cap\T^{b}))\cap\U
=\bigcup_{r\in\mathbb{Z}}(\T^{\leq r}\cap\T^{b}\cap\U)\\
&=\bigcup_{r\in\mathbb{Z}}(\T^{\leq r}\cap\U)
\subseteq \bigcup_{r\in\mathbb{Z}}\mathcal{U}^{\leq r+k}
=\bigcup_{s\in\mathbb{Z}}\mathcal{U}^{\leq s},
\end{align*}
and therefore $(\mathcal{U}^{\le 0},\mathcal{U}^{\ge 1})$ is bounded above. Thus $(2)$ implies $(3)$.

Assume that $(3)$ holds. By $G\in\T^c\subseteq\mathcal{U}$, there exists $k > 0$ such that $G[k] \in \mathcal{U}^{\leq 0}$. Since $\mathcal{U}^{\leq 0}\subseteq \mathcal{U}$ is closed under positive shifts and $(\mathcal{U}^{\le 0},\mathcal{U}^{\ge 1})$ is a $t$-structure on $\U$, we have  $G(-\infty,-k] \subseteq\mathcal{U}^{\leq 0}\subseteq {}^{\perp}(\U^{\geq 1})$. Note that ${}^{\perp}(\U^{\geq 1})\subseteq \T$ is closed under direct summands, coproducts and extensions. Thus $\T^{\leq -k}=\overline{\langle G\rangle}^{(-\infty,-k]}\subseteq{}^{\perp}(\U^{\geq 1})$. This shows that $(3)$ implies $(1)$.
\end{proof}

\begin{lem}\label{lem-lifting-$t$-stucture}
Let $(\mathcal{U}^{\le 0},\mathcal{U}^{\ge 1})$ be a bounded above $t$-structure on $\U$.
Then $(\T^b\cap{^{\perp}(\mathcal{U}^{\ge 1})},\mathcal{U}^{\ge 1})$ is a bounded above $t$-structure on $\mathcal{T}^b$ and there exists some $k>0$ such that $\mathcal{U}^{\geq 1} \subseteq \mathcal{T}^{\geq -k+1} \cap \mathcal{T}^{b}$.
\end{lem}
\begin{proof}
We first prove that $(\T^b\cap{^{\perp}(\mathcal{U}^{\ge 1})},\mathcal{U}^{\ge 1})$ is a $t$-structure on $\T^{b}$. It is straightforward to verify that $(\T^b\cap{^{\perp}(\mathcal{U}^{\ge 1})},\mathcal{U}^{\ge 1})$ satisfies (T1) and (T2) in Definition \ref{defn-$t$-structure}. It suffices to show that $(\T^b\cap{^{\perp}(\mathcal{U}^{\ge 1})},\mathcal{U}^{\ge 1})$ satisfies the axiom (T3).

Let $E \in \mathcal{T}^b$. By Lemma~\ref{lem-strong approximating system}(2)(3), $E$ admits a strong $\Tsb$-approximating system $(E_{*}, f_{*})$ such that $E \simeq \Hocolim(E_{*})$.

By (T3) for the $t$-structure $(\mathcal{U}^{\le 0},\mathcal{U}^{\ge 1})$ on $\U$, we obtain two chains $(E_{*}^{\le 0},f_{*}^{\le 0})$ and $(E_{*}^{\ge 1},f_{*}^{\ge 1})$ in $\U$ with $E_{i}^{\le 0}\in \mathcal{U}^{\le 0}$ and $E_i^{\ge 1}\in \mathcal{U}^{\ge 1}$ for all $i\in \mathbb{N}^{+}$ such that there are commutative squares of morphisms:
\[
\begin{tikzcd}
E_{i}^{\leq 0} \arrow[r] \arrow[d, "f_{i}^{\leq 0}"'] & E_{i} \arrow[r] \arrow[d, "f_{i}"'] & E_{i}^{\geq 1} \arrow[d, "f_{i}^{\geq 1}"'] \arrow[r] & {E_{i}^{\leq 0}[1]} \arrow[d, "{f_{i}^{\leq 0}[1]}"'] \\
E_{i+1}^{\leq 0} \arrow[r]                            & E_{i+1} \arrow[r]                   & E_{i+1}^{\geq 1} \arrow[r]                            & {E_{i+1}^{\leq 0}[1].}
\end{tikzcd}
\]
By Definition \ref{defn-$t$-structure}(T1)(T2), it can be shown that $f_{i}^{\leq 0}$ and $f_{i}^{\geq 1}$ are unique morphisms in $\T$ such that the left and middle squares commute, respectively.
Further, by the $3\times 3$ lemma for triangles in a triangulated category, there exist dotted morphisms making the following diagram of triangles in $\U$ commutative:
$$\xymatrix{
E_i^{\ge 1}[-1] \ar[d] \ar[r]
& E_i^{\le 0} \ar[r] \ar@{-->}[d]_-{f_{i}^{\leq 0}}
& E_i \ar[r] \ar[d]_-{f_{i}}
& E_i^{\ge 1} \ar[d]_-{ f_{i}^{\geq 1}} \\
E_{i+1}^{\ge 1}[-1] \ar[d] \ar[r]
& E_{i+1}^{\le 0} \ar[r] \ar@{-->}[d]
& E_{i+1} \ar[r] \ar[d]
& E_{i+1}^{\ge 1} \ar[d] \\
D''_i[-1] \ar@{-->}[r]^{a_i}
& D'_i \ar@{-->}[r]^{b_i}\ar@{-->}[d]
& D_i \ar@{-->}[r]^{c_i}\ar[d]
& D''_i\ar[d]\\
& E_i^{\le 0}[1] \ar[r]
& E_i[1] \ar[r]
& E_i^{\ge 1}[1].
}$$
Note that $E_{i+1}^{\le 0}\in \mathcal{U}^{\le 0}$ and $E_i^{\le 0}[1]\in \mathcal{U}^{\le -1}\subseteq \mathcal{U}^{\le 0}$. Since $\mathcal{U}^{\le 0}$ is closed under extensions, $D'_i\in \mathcal{U}^{\le 0}$.
Similarly, we obtain $D''_i\in \mathcal{U}^{\ge 0}$, due to $E_{i+1}^{\ge 1},E_i^{\ge 1}[1]\in \mathcal{U}^{\ge 0}$. As $(\mathcal{U}^{\le 0},\mathcal{U}^{\ge 1})$ is bounded above, we see from Lemma~\ref{lem-bounded-above}(2) that $\T^{\leq -k}\cap\U\subseteq\U^{\leq 0}$ for some integer $k>0$. Now, we consider $i\geq k$. By Corollary~\ref{lem-cauchy seq}(1), $D_i\in\T^{\leq -i-1}\cap\U\subseteq\U^{\leq -1}$. It follows from $\Hom_{\mathcal{T}}(\mathcal{U}^{\le -1},\mathcal{U}^{\ge 0})=0$ that $\Hom_{\mathcal{T}}(D_i,D''_i)=0$. This implies $c_i=0$, and therefore $D'_i\simeq D''_i[-1]\oplus D_i$.
Note that $D'_i\in \mathcal{U}^{\le 0}$, $D''_i[-1]\in \mathcal{U}^{\ge 1}$ and $\Hom_{\mathcal{T}}(\mathcal{U}^{\le 0},\mathcal{U}^{\ge 1})=0$.
This forces $\Hom_{\mathcal{T}}(D'_i,D''_i[-1])=0$. Thus $D''_i[-1]=0$ (equivalently, $D''_i=0$) and $f_i^{\geq 1}$ is an isomorphism for $i\geq k$. It follows that $\Hocolim(E_{*}^{\geq 1})\simeq E_{k}^{\ge 1}\in \mathcal{U}^{\ge 1}.$

Again by the $3\times 3$ lemma for triangles in a triangulated category, there exists an object $Y\in \mathcal{T}$ and dotted morphisms making the following diagram of triangles in $\mathcal{T}$ commutative:
$$\xymatrix{
\bigoplus_{i\ge 1}E_i^{\ge 1}[-1] \ar[d] \ar[r]
&\bigoplus_{i\ge 1} E_i^{\le 0} \ar[r] \ar@{-->}[d]
& \bigoplus_{i\ge 1}E_i \ar[r] \ar[d]_-{(1-f_*)}
& \bigoplus_{i\ge 1}E_i^{\ge 1} \ar[d]_-{(1-f_*^{\geq 1})} \\
\bigoplus_{i\ge 1}E_i^{\ge 1}[-1] \ar[d] \ar[r]
&\bigoplus_{i\ge 1} E_i^{\le 0} \ar[r] \ar@{-->}[d]
&  \bigoplus_{i\ge 1}E_i \ar[r] \ar[d]
& \bigoplus_{i\ge 1}E_i^{\ge 1} \ar[d] \\
E_{k}^{\ge 1}[-1] \ar@{-->}[r]
& Y \ar@{-->}[r] \ar@{-->}[d]
& \Hocolim(E_{*}) \ar@{-->}[r] \ar[d]
& \Hocolim(E_{*}^{\geq 1})\ar[d]\\
& \bigoplus_{i\ge 1} E_i^{\le 0}[1] \ar[r]
& \bigoplus_{i\ge 1}E_i[1] \ar[r]
& \bigoplus_{i\ge 1}E_i^{\ge 1}[1].
}$$
Note that $\Hocolim(E_{*}^{\geq 1})\simeq E_{k}^{\ge 1}\in \mathcal{U}^{\ge 1}\subseteq \mathcal{T}^b$ and $\Hocolim(E_{*}) \simeq E\in \mathcal{T}^b$. Consequently, there is a triangle in $\mathcal{T}^b$:
$$ E_{k}^{\ge 1}[-1]\lra Y \lra E\lra E_{k}^{\ge 1}.$$
To verify the axiom (T3) for $(\T^{b}\cap{}^{\perp}(\mathcal{U}^{\ge 1}), \mathcal{U}^{\ge 1})$, we only need to show $Y\in {^{\perp}(\mathcal{U}^{\ge 1})}$.

In fact, since $\Hom_{\mathcal{T}}(\mathcal{U}^{\le 0},\mathcal{U}^{\ge 1})=0$, we have  $\Hom_{\mathcal{T}}(E_i^{\le 0},\mathcal{U}^{\ge 1})=0$ and $\Hom_{\mathcal{T}}(E_i^{\le 0}[1],\mathcal{U}^{\ge 1})=0$ for all $i\ge 1$.
Then $\Hom_{\mathcal{T}}(\bigoplus_{i\ge 1}E_i^{\le 0},\mathcal{U}^{\ge 1})\simeq \prod_{i\ge 1}\Hom_{\mathcal{T}}(E_i^{\le 0},\mathcal{U}^{\ge 1})=0$ and $\Hom_{\mathcal{T}}(\bigoplus_{i\ge 1}E_i^{\le 0}[1],\mathcal{U}^{\ge 1})\simeq \prod_{i\ge 1}\Hom_{\mathcal{T}}(E_i^{\le 0}[1],\mathcal{U}^{\ge 1})=0$.
Thus $\Hom_{\mathcal{T}}(Y,\mathcal{U}^{\ge 1})=0$ by the second column of the above diagram.

Up to now, we have shown that $(\T^{b}\cap{}^{\perp}(\mathcal{U}^{\ge 1}),\mathcal{U}^{\ge 1})$ is a $t$-structure on $\mathcal{T}^b$. In the following, we show $\mathcal{U}^{\geq 1} \subseteq \mathcal{T}^{\geq -k+1} \cap \mathcal{T}^{b}$. Since $\T^{\leq -k}\cap\U\subseteq\U^{\leq 0}$, it is clear that $\U^{\geq 1}= (\mathcal{U}^{\le 0})^{\perp}\cap\U\subseteq (\mathcal{T}^{\le -k}\cap \U)^{\perp}\cap\T^{b}.$ So, it suffices to show $(\mathcal{T}^{\le -k}\cap \U)^{\perp}\cap\T^{b}=\T^{\geq -k+1}\cap\T^{b}$.

Actually, by the assumption $\Tsb\subseteq \U\subseteq \T^{b}$, we have $\T^{\leq -k}\cap \Tsb \subseteq \T^{\leq -k}\cap\U \subseteq \T^{\leq -k}\cap \T^{b}.$ Then $$\T^{\geq -k+1}\cap\T^{b}=(\T^{\leq -k}\cap \T^{b})^{\perp} \cap\T^{b} \subseteq (\T^{\leq -k}\cap\U)^{\perp} \cap\T^{b} \subseteq (\T^{\leq -k}\cap\Tsb)^{\perp}\cap\T^{b},$$
where the first identity follows from the $t$-structure $(\T^{\leq 0}\cap\T^{b},\T^{\geq 1}\cap\T^{b})$.
So, it is enough to show
$$(\diamondsuit):\quad (\T^{\leq -k}\cap\Tsb)^{\perp}\cap\T^{b}\subseteq (\T^{\leq -k}\cap\T^{b})^{\perp}\cap\T^{b}.$$

For this aim, let $Y\in(\T^{\leq -k}\cap\Tsb)^{\perp}\cap\T^{b}$ and $X\in\T^{\leq -k}\cap\T^{b}$. Then there exists an integer $r\geq k$ such that $Y\in\T^{\geq -r}$. By Lemma~\ref{lem-strong approximating system}(3), $X$ admits a strong $\Tsb$-approximating system $(X_{*}, h_{*})$ with $X \simeq \Hocolim(X_{*})$.
Further, by Corollary~\ref{lem-cauchy seq}(2), we obtain a triangle
$L_{r} \to X_{r} \to X \to L_{r}[1]$ in $\T$ with $L_{r} \in \mathcal{T}^{\leq -r} \subseteq \mathcal{T}^{\leq -k}$. This implies $\Hom_{\mathcal{T}}(L_{r}[1], Y) = 0$.
Since $X \in \mathcal{T}^{\leq -k}$, we have $ X_r\in\mathcal{T}^{\leq -k}\cap \Tsb$.
Note that $Y\in(\T^{\leq -k}\cap\Tsb)^{\perp}$, and further $\Hom_{\mathcal{T}}(X_{r}, Y) = 0$.
It follows that $\Hom_{\mathcal{T}}(X, Y) = 0$. This shows $(\diamondsuit)$, and thus $\mathcal{U}^{\geq 1} \subseteq \mathcal{T}^{\geq -k+1} \cap \mathcal{T}^{b}$.

Finally, we prove that $(\T^{b}\cap{}^{\perp}(\mathcal{U}^{\ge 1}), \mathcal{U}^{\ge 1})$ is bounded above.

Note that $(\T^{\leq -k}\cap\T^{b},\T^{\geq -k+1}\cap\T^{b})$ is a bounded $t$-structure on $\T^b$.
Together with $\mathcal{U}^{\geq 1} \subseteq \mathcal{T}^{\geq -k+1} \cap \mathcal{T}^{b}$, we obtain
$\T^{\leq -k}\cap\T^{b}={}^{\perp}(\T^{\geq -k+1}\cap\T^{b})\cap\T^{b}\subseteq{}^{\perp}(\U^{\geq 1})\cap\T^{b}$. Thus $(\T^{b}\cap{}^{\perp}(\mathcal{U}^{\ge 1}),\mathcal{U}^{\ge 1})$ is a bounded above $t$-structure on $\mathcal{T}^b$.
\end{proof}

\begin{lem}\label{lem-equi-$t$-structure}
Let $(\mathcal{U}^{\le 0},\mathcal{U}^{\ge 1})$ be a bounded $t$-structure on $\U$. Suppose that $\mathcal{T}$ has finite strong finitistic dimension (see Definition~{\rm \ref{defn-fd2}}). Then the $t$-structure $(\T^{b}\cap{}^{\perp}(\mathcal{U}^{\ge 1}),\mathcal{U}^{\ge 1})$ on $\T^b$ in Lemma {\rm \ref{lem-lifting-$t$-stucture}} is equivalent to the bounded $t$-structure $(\mathcal{T}^{\le 0}\cap \mathcal{T}^b,\mathcal{T}^{\ge 1}\cap \mathcal{T}^b)$ on $\T^b$. In particular, $(\T^{b}\cap{}^{\perp}(\mathcal{U}^{\ge 1}),\mathcal{U}^{\ge 1})$ is a bounded $t$-structure on $\T^{b}$.
\end{lem}
\begin{proof}
By Lemma~\ref{lem-lifting-$t$-stucture}, there exists an integer $n>0$ such that $\mathcal{U}^{\geq 1} \subseteq \T^{\geq -n+1} \cap \T^{b}$.
It suffices to prove that there exists an integer $m$ such that $\T^{\geq -m+1}\cap\T^{b}\subseteq\mathcal{U}^{\geq 1}$.

Since $(\mathcal{U}^{\le 0},\mathcal{U}^{\ge 1})$ is a bounded $t$-structure on $\mathcal{U}$ and $G\in\T^c\subseteq \mathcal{U}$, we have $G\in \mathcal{U}^{\ge s}$ for some integer $s$.
Then $G[1,+\infty)\subseteq \mathcal{U}^{\ge s+1}$. Moreover, since $\T$ has finite strong finitistic dimension, there exists an integer $B$ such that $\Tsb\cap{}^{\perp}(G[1,+\infty))\subseteq \bigcup_{A\le B}\overline{\langle G\rangle}^{[A,\,B]}$. Consequently,
\[
\Tsb\cap{}^{\perp}(\mathcal{U}^{\geq s+1})\subseteq \Tsb\cap{}^{\perp}(G[1,+\infty))\subseteq \bigcup_{A\le B}\overline{\langle G\rangle}^{[A,\,B]}.
\vspace{-1.0em}\]
It follows from $(\bigcup_{A\le B}\overline{\langle G\rangle}^{[A,\,B]})^{\perp}=G(-\infty,B]^{\perp}$ that $G(-\infty,B]^{\perp}\cap\T^b\subseteq(\Tsb\cap{}^{\perp}(\mathcal{U}^{\ge s+1}) )^{\perp}\cap\T^{b}.$
By our convention, we take $(\T^{\leq 0},\T^{\geq 1}) = (\T_{G}^{\leq 0},\T_{G}^{\geq 1})$. By Lemma \ref{lem-p-w-gen}, $\T^{\geq B+1}=G(-\infty,B]^{\perp}$. Thus
$$
(\ast)\quad \T^{\geq B+1}\cap\T^b\subseteq(\Tsb\cap{}^{\perp}(\mathcal{U}^{\ge s+1}) )^{\perp}\cap\T^{b}.
$$
In the following, we show
\[
(\ast\ast)\quad(\Tsb\cap{}^{\perp}(\mathcal{U}^{\ge s+1}) )^{\perp}\cap\T^{b}\subseteq (\T^b\cap{}^{\perp}(\mathcal{U}^{\ge s+1}) )^{\perp}\cap\T^{b}.
\]
Actually, this inclusion is an equality since the right side is always included in the left side.

For this aim, let $N\in(\Tsb\cap{}^{\perp}(\mathcal{U}^{\ge s+1}) )^{\perp}\cap\T^{b}$ and $M\in\T^b\cap{}^{\perp}(\mathcal{U}^{\ge s+1}) $.
Since $(\mathcal{U}^{\le 0},\mathcal{U}^{\ge 1})$ is a bounded above $t$-structure on $\mathcal{U}$, there exists an integer $k>0$ such that $\T^{\leq-(k-s)}\subseteq{}^{\perp}(\U^{\geq s+1})$ by Lemma~\ref{lem-bounded-above}(1). By $N \in \mathcal{T}^{b}$, we have $N \in \mathcal{T}^{\geq -r}$ for some integer $r\geq \max\{1, k-s\}$. Moreover, by Lemma~\ref{lem-strong approximating system}(3), $M$ admits a strong $\Tsb$-approximating system $(M_{*}, f_{*})$ with $M \simeq \Hocolim(M_{*})$. Combining this with  Corollary \ref{lem-cauchy seq}(2), we obtain a triangle
$L_{r} \to M_{r} \to M \to L_{r}[1]$ in $\T$ with $L_{r} \in \mathcal{T}^{\leq -r} \subseteq \mathcal{T}^{\leq -(k-s)}\subseteq{}^{\perp}(\U^{\geq s+1})$. It follows from $M\in\T^b\cap{}^{\perp}(\mathcal{U}^{\ge s+1}) $ that $M_{r}\in\Tsb\cap{}^{\perp}(\U^{\geq s+1})$.
This implies $\Hom_{\mathcal{T}}(M_{r}, N) = 0$, due to $N\in(\Tsb\cap{}^{\perp}(\mathcal{U}^{\ge s+1}) )^{\perp}$. Note that $\Hom_{\mathcal{T}}(L_{r}[1], N) = 0$ since $L_{r}[1] \in \mathcal{T}^{\leq -r-1}$ and $N\in\T^{\geq -r}$. Thus $\Hom_{\mathcal{T}}(M, N) = 0$. This shows $(\ast\ast)$.

By ($\ast$) and ($\ast\ast$), $\T^{\geq B+1}\cap\T^b\subseteq(\T^b\cap{}^{\perp}(\mathcal{U}^{\ge s+1}) )^{\perp}\cap\T^{b}$. Since $(\T^{b}\cap{}^{\perp}(\mathcal{U}^{\ge 1}),\mathcal{U}^{\ge 1})$ is a $t$-structure on $\mathcal{T}^b$ by Lemma \ref{lem-lifting-$t$-stucture}, we have $(\T^b\cap{}^{\perp}(\mathcal{U}^{\ge s+1}) )^{\perp}\cap\T^{b}=\mathcal{U}^{\geq s+1}$. This implies $\T^{\geq B+1}\cap\T^b\subseteq\mathcal{U}^{\geq s+1}$. Let $m:=s-B$. Then $\mathcal{T}^{\geq -m+1}\cap\T^b\subseteq\mathcal{U}^{\geq 1}$.
Thus $\T^{\geq -m+1}\cap\T^{b}\subseteq\mathcal{U}^{\geq 1} \subseteq \T^{\geq -n+1} \cap \T^{b}$.
This shows that the $t$-structures $(\T^{b}\cap{}^{\perp}(\mathcal{U}^{\ge 1}),\mathcal{U}^{\ge 1})$ and  $(\mathcal{T}^{\le 0}\cap \mathcal{T}^b,\mathcal{T}^{\ge 1}\cap \mathcal{T}^b)$ are equivalent.
\end{proof}

{\bf Proof of Theorem \ref{main1}(2).} Assume that $\mathcal{U}$ admits a bounded $t$-structure $(\mathcal{U}^{\le 0},\mathcal{U}^{\ge 1})$. By Lemma~\ref{lem-lifting-$t$-stucture}, $(\T^b\cap{}^\perp(\mathcal{U}^{\ge 1}),\mathcal{U}^{\ge 1})$ is a $t$-structure on $\mathcal{T}^b$. Since $\T$ has finite strong finitistic dimension, it follows from  Lemma \ref{lem-equi-$t$-structure} that $(\T^b\cap{}^\perp(\mathcal{U}^{\ge 1}),\mathcal{U}^{\ge 1})$ is a bounded $t$-structure on $\mathcal{T}^b$.
Thus $\mathcal{T}^b=\bigcup_{v\in \mathbb{Z}}\mathcal{U}^{\ge v}=\mathcal{U}$.
But this contradicts the assumption $\mathcal{U}\subsetneq\mathcal{T}^b$.
 \hfill$\square$

\section{Triangulated categories with finite strong finitistic dimension}\label{4}

In this section, we first introduce the concept of Gorenstein triangulated categories and then prove Proposition \ref{FGD} and Corollary \ref{mainapp}. Finally, we construct triangulated categories with finite strong finitistic dimension via recollements  (Proposition~\ref{construct1}). Consequently, Theorem \ref{main1} can also be applied to non-Gorenstein triangulated categories (Corollary~\ref{app2}).

\begin{defn}\label{Gorenstein}
Let $R$ be a commutative ring and $\T$ an $R$-linear, compactly generated triangulated category with a compact generator $G$. Denote by $E$ the Brown-Comenetz dual of $G$ (with respect to $I$), as in Definition {\rm \ref{def bcdual}}. We say that $\T$ is \emph{Gorenstein} if $\langle G\rangle = \langle E\rangle$, that is, $G$ and $E$ generate the same thick subcategory of $\T$.

\end{defn}

By \cite[Section 3]{CSZ26}, whether a monogenic triangulated category $\T$ is Gorenstein is independent of the choice of its compact generators. We note that the notion of a Gorenstein triangulated category was also introduced in \cite[Definition~4.2(ii)]{KPV25} for compactly generated categories, formulated in terms of orthogonality relations between subcategories. The precise relationship between their definition and ours remains unclear.

In the following, we construct examples of Gorenstein triangulated categories from both differential graded (DG) algebras and schemes. Before this, we recall some standard notation and definitions.

Let $R$ be a \emph{commutative Artin ring} and $D \coloneqq \Hom_R(-,J)$ the standard duality, where $J$ denotes the injective envelope of the direct sum of all simple $R$-modules (up to isomorphism). Let $A$ be a DG $R$-algebra. We denote by $\D{A}$ the derived category of $A$ (see \cite{K94}), which is an $R$-linear, compactly generated triangulated category with the compact generator $A$. Note that every Artin algebra $\Lambda$ can be regarded as a DG algebra concentrated in degree $0$, and $\D{\Lambda}$ coincides with the unbounded derived category $\D{\Lambda\Modcat}$ of $\Lambda$.

Let $X$ be a quasi-compact, quasi-separated scheme with a morphism $f\colon X\to\Spec(R)$. Then the derived functor $\R f_*\colon \mathscr{D}_{\mathrm{qc}}(X) \to \D{R\Modcat}$ admits a left adjoint $\mathbb{L} f^*$ and a right adjoint $f^\times$ (see \cite[Lemmas 20.28.1, 36.3.8 and 48.3.1]{S24}). We denote by $\mathscr{D}_{\mathrm{qc}}(X)$ the full subcategory of the unbounded derived category of $\mathscr{O}_X$-modules consisting of (cochain) complexes of $\mathscr{O}_X$-modules with \emph{quasicoherent} cohomology. By \cite[Theorem 3.1.1(ii)]{BM03}, $\mathscr{D}_{\mathrm{qc}}(X)$ is a compactly generated triangulated category with a single compact generator. Further, we define
$\mathscr{D}^{\mathrm{perf}}(X)$ to be the full subcategory of $\mathscr{D}_{\mathrm{qc}}(X)$ consisting of all perfect complexes, where a complex is said to be \emph{perfect} if it is locally isomorphic to a bounded complex of finite-rank vector bundles.

\begin{lem}\label{BCdual}
$(1)$ Let $A$ be a DG $R$-algebra. Then the Brown-Comenetz dual of $A$ in $\D{A}$ is $D(A)$.

$(2)$ Let $G\in\mathscr{D}^{\mathrm{perf}}(X)$. Then the Brown-Comenetz dual of $G$ in $\mathscr{D}_{\mathrm{qc}}(X)$ is $G\otimesL_{\mathscr{O}_{X}} f^\times(J)$.
\end{lem}
\begin{proof}
$(1)$ Let $M\in\mathscr{D}(A)$. Note that there are canonical isomorphisms\vspace{-0.3em}\[\Hom_{\mathscr{D}(A)}(M,D(A))\simeq
\Hom_{\mathscr{D}(A)}(M,\rHom_R(A,J))\simeq\Hom_{\mathscr{D}(R)}(A\otimesL_{A} M,J)\simeq\Hom_{\mathscr{D}(R)}(M,J)\simeq H^0(D(M)).\]
Since $D$ is exact, it commutes with taking cohomology. Then we obtain $H^0(D(M))\simeq D(H^0(M))\simeq D\Hom_{\mathscr{D}(A)}(A,M).$ It follows that $D\Hom_{\mathscr{D}(A)}(A,-)\simeq\Hom_{\mathscr{D}(A)}(-,D(A))$. Thus the Brown-Comenetz dual of $A$ is $D(A)$.

$(2)$ Let $G^{\vee}\coloneqq\rHom_{\mathscr{O}_X}(G,\mathscr{O}_X)$. Then $G^{\vee}\in\mathscr{D}^{\mathrm{perf}}(X)$ and $G\simeq (G^{\vee})^{\vee}$. We first show that, for any $M\in\mathscr{D}_{\mathrm{qc}}(X)$, there exists an isomorphism of $R$-modules:
\[\Hom_{\mathscr{D}_{\mathrm{qc}}(X)}(G,M)\simeq H^0(\R f_*(G^{\vee}\otimesL_{\mathscr{O}_X} M)).\]
In fact, there are canonical isomorphisms
\[
\Hom_{\mathscr{D}_{\mathrm{qc}}(X)}(G,M)\simeq\Hom_{\mathscr{D}_{\mathrm{qc}}(X)}(\mathscr{O}_X\otimesL_{\mathscr{O}_X}G,M)
\simeq\Hom_{\mathscr{D}_{\mathrm{qc}}(X)}(\mathscr{O}_X,\rHom_{\mathscr{O}_X}(G,M)).
\]
Note that $(\mathbb{L} f^*,\R f_*)$ is an adjoint pair with $\mathbb{L} f^*(R)=\mathscr{O}_X$, and that $\rHom_{\mathscr{O}_X}(G,M)\simeq G^{\vee}\otimesL_{\mathscr{O}_X} M$ by \cite[Lemma 20.50.5]{S24}.
Thus
\begin{align*}
\Hom_{\mathscr{D}_{\mathrm{qc}}(X)}(G,M)\simeq\Hom_{\mathscr{D}_{\mathrm{qc}}(X)}(\mathscr{O}_X,\rHom_{\mathscr{O}_X}(G,M))
&= \Hom_{\mathscr{D}_{\mathrm{qc}}(X)}(\mathbb{L} f^*(R),\rHom_{\mathscr{O}_X}(G,M)) \\
&\simeq \Hom_{\D{R\text{-}{\rm Mod}}}(R,\R f_*(\rHom_{\mathscr{O}_X}(G,M))) \\
&\simeq H^0(\R f_*(G^{\vee}\otimesL_{\mathscr{O}_X} M)).
\end{align*}
Clearly, $D$ is exact and commutes with taking cohomology. Then
\[
D\Hom_{\mathscr{D}_{\mathrm{qc}}(X)}(G,M)\simeq  \Hom_R(H^0(\R f_*(G^{\vee}\otimesL_{\mathscr{O}_X} M)),J)\simeq H^0(\Hom_R(\R f_*(G^{\vee}\otimesL_{\mathscr{O}_X} M),J)).
\]
As $J$ is an injective $R$-module, we have
\[
H^0(\Hom_R(\R f_*(G^{\vee}\otimesL_{\mathscr{O}_X} M),J))\simeq\Hom_{\D{R\text{-}{\rm Mod}}}(\R f_*(G^{\vee}\otimesL_{\mathscr{O}_X} M),J).
\]
Since both $(\R f_*,f^\times)$ and $(G^{\vee}\otimesL_{\mathscr{O}_X}-,\rHom_{\mathscr{O}_X}(G^{\vee},-))$ are adjoint pairs, it follows that
\begin{align*}
D\Hom_{\mathscr{D}_{\mathrm{qc}}(X)}(G,M)
&\simeq \Hom_{\D{R\text{-}{\rm Mod}}}(\R f_*(G^{\vee}\otimesL_{\mathscr{O}_X} M),J) \\
&\simeq \Hom_{\mathscr{D}_{\mathrm{qc}}(X)}(G^{\vee}\otimesL_{\mathscr{O}_X} M,f^\times(J))\\
&\simeq\Hom_{\mathscr{D}_{\mathrm{qc}}(X)}(M,\rHom_{\mathscr{O}_X}(G^{\vee},f^\times(J))).
\end{align*} Still by \cite[Lemma 20.50.5]{S24}, there is a natural isomorphism $\rHom_{\mathscr{O}_X}(G^{\vee},-)\simeq G\otimesL_{\mathscr{O}_X}-$. Thus $$D\Hom_{\mathscr{D}_{\mathrm{qc}}(X)}(G,M)\simeq\Hom_{\mathscr{D}_{\mathrm{qc}}(X)}(M,G\otimesL_{\mathscr{O}_X} f^\times(J)).$$
This shows that the Brown-Comenetz dual of $G$ in $\mathscr{D}_{\mathrm{qc}}(X)$ is $G\otimesL_{\mathscr{O}_{X}} f^\times(J)$.
\end{proof}

Now, we are in a position to give examples of Gorenstein triangulated categories.

\begin{exam}\label{Gexam}
(1) Let $A$ be a DG $R$-algebra. Following \cite[Section~0]{J20}, $A$ is called \emph{Gorenstein} if $\langle A \rangle = \langle D(A) \rangle$ in  $\mathscr{D}(A)$. It follows from Lemma \ref{BCdual}(1) that $\D{A}$ is Gorenstein. In particular, if $A$ is a Gorenstein Artin algebra or a $d$-self-injective DG algebra (see \cite[Definition~2.2]{J20}), then $\mathscr{D}(A)$ is Gorenstein.

(2) Suppose that $R$ is a commutative Gorenstein Artin ring and $f:X\to\Spec(R)$ is a proper, \emph{Gorenstein} morphism (that is, $f$ is flat and its fibres are Gorenstein schemes). We show that $\mathscr{D}_{\mathrm{qc}}(X)$ is Gorenstein.

In fact, since $f$ is proper, we see from \cite[Section 48.19]{S24} that the restriction
of the functor $f^\times: \mathscr{D}(R\Modcat)\to \mathscr{D}_{\mathrm{qc}}(X)$ to
$\mathscr{D}^{+}(R\Modcat)$ coincides with another triangle functor $f^!:\mathscr{D}^+(R\Modcat)\to \mathscr{D}^{+}_{\mathrm{qc}}(X)$ (see the discussion below \cite[Situation 48.16.1]{S24}). As $f$ is a proper Gorenstein morphism, the object $f^!(R)\in \mathscr{D}_{\mathrm{qc}}(X)$ is invertible by \cite[Lemma 48.25.10]{S24}. Further, by \cite[Lemma 20.52.2]{S24}, $f^!(R)\in\mathscr{D}^{\mathrm{perf}}(X)$. Thus the functor $-\otimesL_{\mathscr{O}_X}f^!(R)$ induces an autoequivalence of $\mathscr{D}^{\mathrm{perf}}(X)$. Now, let $G$ be a compact generator of $\mathscr{D}_{\mathrm{qc}}(X)$. Then
$\mathscr{D}^{\mathrm{perf}}(X)=\langle G\rangle=\langle G\otimesL_{\mathscr{O}_X}f^!(R)\rangle=\langle G\otimesL_{\mathscr{O}_X}f^\times(R)\rangle.$
Since $R$ is a commutative Gorenstein Artin ring, $R\simeq J$ as $R$-modules. It follows that
$\langle G\rangle=\langle G\otimesL_{\mathscr{O}_X}f^\times(J)\rangle.$ By Lemma \ref{BCdual}(2), $\mathscr{D}_{\mathrm{qc}}(X)$ is Gorenstein.

Note that $\mathscr{D}_{\mathrm{qc}}(X)$ is a weakly approximable triangulated category with a bounded compact generator (see \cite[Summary 5.9]{N18}). Thus Theorem \ref{main1} can be applied to $\mathscr{D}_{\mathrm{qc}}(X)$.

\end{exam}

{\bf Proof of Proposition \ref{FGD}.} (1) Let $G$ be a bounded compact generator of $\T$ and $E$ the Brown-Comenetz dual of $G$. Since $\T$ is Gorenstein, we have $\langle G\rangle = \langle E\rangle$. This implies that $E$ is also a compact generator of $\T$. Note that the finiteness of strong finitistic dimension for $\T$ is independent of the choice of compact generators of $\T$. So, it suffices to show that
there exists an integer $m$ such that
$$\T^{sb}\cap {}^{\bot}(E[1, +\infty))\subseteq\bigcup_{n\le m}\overline{\langle E\rangle}^{[n,\,m]}.\vspace{-0.8em}$$
Let $X\in\T^{sb}\cap {}^{\bot}(E[1, +\infty))$. Since $D\Hom_{\T}(G[i],X)\simeq\Hom_{\T}(X,E[i])$ for $i\in\mathbb{Z}$, we have $X\in G[1,+\infty)^{\perp}$. As $\T$ is weakly approximable, we can fix an integer $A>0$ in Definition~\ref{defn-app}. It follows from \cite[Lemma~3.9(iv)]{BNP23} that $G[-2A,+\infty)^{\perp}\subseteq \mathcal{T}^{\leq -A-1}$. By $X[2A+1]\in G[-2A,+\infty)^{\perp}$, we have   $X[A]\in\T^{\leq 0}$. Further, by Lemma \ref{lem-app-induction}, for each $i>0$, there exists a triangle $C_i\to X[A]\to D_i\to C_i[1]$ in $\T$ with $C_i\in\overline{\langle G\rangle}^{[1-i-A,\,A]}$ and $D_i\in\T^{\leq -i}$. Since $\Tsb=\bigcup_{m\in\mathbb{Z}}\big(\mathcal{T}^-\cap {}^{\perp}(\mathcal{T}^{\le m})\big)$ by \cite[Proposition 3.4]{CCZ26}, it follows from
$X\in\T^{sb}$ that there is an integer $b>0$ such that $X[A]\in{}^{\perp}(\T^{\leq -b})$. This implies $\Hom_\T(X[A], D_b)=0$, and therefore $X[A]$ is a direct summand of $C_b$. Since $C_b\in\overline{\langle G\rangle}^{[1-b-A,\,A]}$, we have $X\in\overline{\langle G\rangle}^{[1-b,\,2A]}$. Note that $G\in\langle E\rangle^{[-k,\,k]}$ for some $k>0$, due to $\langle G\rangle = \langle E\rangle$. It follows that  $\overline{\langle G\rangle}^{[1-b,\,2A]}\subseteq\overline{\langle E\rangle}^{[1-b-k,\,2A+k]}$. Clearly, neither $A$ nor $k$ depends on $X$. Now, let $m:=2A+k$. Then $\T^{sb}\cap{}^{\perp}E[1,+\infty)\subseteq\bigcup_{n\leq m}\overline{\langle E\rangle}^{[n,\,m]}$.
This shows that $\T$ has finite strong finitistic dimension.

(2) We denote by $(\mathcal{T}^c)^{\mathrm{op}}$ the opposite category of $\mathcal{T}^c$, and write $\Sigma$ for the shift functor of $(\mathcal{T}^c)^{\mathrm{op}}$. Let $X^{{\rm op}} \in(\mathcal{T}^c)^{\rm op}$ satisfy $\Hom_{(\mathcal{T}^c)^{\mathrm{op}}}(G^{{\rm op}}, \Sigma^n X^{\rm op})=0$ for all $n \le -1$. It follows that $X\in \mathcal{T}^c\cap {}^{\perp}G[1,+\infty)$. Since $\T$ has finite strong finitistic dimension, there exists an integer $t$ such that $\T^{sb}\cap {}^{\bot}(G[1, +\infty))\subseteq\bigcup_{s\le t}\overline{\langle G\rangle}^{[s,\,t]}$. Combining this with $\T^c\subseteq \T^{sb}$, we have
$$X\in\T^c\cap\big(\bigcup_{s\leq t}\overline{\langle G\rangle}^{[s,\,t]}\big)=\bigcup_{s\leq t}\big(\T^c\cap\overline{\langle G\rangle}^{[s,\,t]}\big).$$
By \cite[Lemma 1.8(ii)]{N21}, $\T^c\cap\overline{\langle G\rangle}^{[s,\,t]}\subseteq\langle G\rangle^{[s,\,t]}$ for all $s\leq t$. This implies that $X\in\bigcup_{s\leq t}\langle G\rangle^{[s,\,t]}$; equivalently,
$X^{{\rm op}}\in\bigcup_{s\leq t}\langle G^{{\rm op}}\rangle^{[-t,\,-s]}$.
Thus $(\mathcal{T}^c)^{\mathrm{op}}$ has finite finitistic dimension.
\hfill$\square$

\vspace{1em}

To show Corollary \ref{mainapp}, we establish the following result.

\begin{lem}\label{lasttech}
Let $\T$ be a weakly approximable triangulated category with a bounded compact generator. If $\Tsb=\T^b$, then $\T^c=\T_c^b$. The converse holds if either $\T = \mathscr{D}_{\mathrm{qc}}(X)$ for a finite-dimensional noetherian scheme $X$ or $\T=\D{\Lambda\Modcat}$ for an Artin algebra $\Lambda$.
\end{lem}
\begin{proof}
We first show $\T^c = \Tsb \cap \T_c^b$.

For this aim, let $G$ be a bounded compact generator of $\T$. Since $\T^c = \langle G \rangle$
and $G$ is bounded,  it follows that $\T^c \subseteq \Tsb \cap \T_c^b$. To show the converse, let $X\in\Tsb\cap\T_{c}^b$. Since $\T$ is weakly approximable, there exists a positive integer $m$ with $X\in{}^{\perp}(\T_{G}^{\leq -m})$ by \cite[Proposition 3.4]{CCZ26}. Further, by the definition of $\T_{c}^{b}$, there is a triangle $C_{m}\to X\to D_{m}\to C_{m}[1]$ with $C_{m}\in\T^{c}$ and $D_{m}\in\T_{G}^{\leq -m}$. Then $\Hom_\T(X, D_m)=0$, and therefore $X$ is a direct summand of $C_m$. This shows $X\in\T^{c}$. Thus $\T^c = \Tsb \cap \T_c^b$. If $\Tsb=\T^b$, then $\T^c = \T^b \cap \T_c^b=\T_c^b$.

If $\T=\D{\Lambda\Modcat}$ for an Artin algebra $\Lambda$ and $\T^c=\T_c^b$ (i.e. $\mathscr{K}^b(\prj{\Lambda})=\mathscr{D}^b(\Lambda\modcat))$, then $\Lambda$ has finite global dimension. In this case, $\mathscr{K}^b(\Lambda\text{-}{\rm Proj})=\mathscr{D}^b(\Lambda\Modcat)$, namely, $\Tsb=\T^b$.

Now, we consider $\T = \mathscr{D}_{\mathrm{qc}}(X)$ for a finite-dimensional noetherian scheme $X$. Then $X$ can be covered by affine open subsets, say $X=\bigcup_{1\leq i\leq n\in\mathbb{N}}\Spec(R_i)$, where each $R_i$ is a commutative noetherian ring of finite Krull dimension. Assume $\T^c=\T_c^b$.
By \cite[Remark 5.8]{N18}, $\T^c=\mathscr{D}^{\mathrm{perf}}(X)$ and $\T_c^b=\mathscr{D}_{\mathrm{coh}}^b(X)$. This forces $\mathscr{D}^{\mathrm{perf}}(X)=\mathscr{D}_{\mathrm{coh}}^b(X)$, and therefore  $\mathscr{D}^{\mathrm{perf}}(\Spec(R_i))=\mathscr{D}_{\mathrm{coh}}^b(\Spec(R_i))$ for all $i$. Clearly, $\mathscr{D}^{\mathrm{perf}}(\Spec(R_i))=\Kb{\prj{R_i}}$ and $\mathscr{D}_{\mathrm{coh}}^b(\Spec(R_i))=\Db{R_i\modcat}$. Thus $\Kb{\prj{R_i}}=\Db{R_i\modcat}$; in other words, $R_i$ is regular. Since $\dim(R_i)\leq\dim(X)<+\infty$, each $R_i$ has finite global dimension. Now, by \cite[Theorem 2.1]{N21}, $\mathscr{D}_{qc}(X)$ is strongly compactly generated. Further, by \cite[Theorem 5.9 and Lemma 5.2(2)]{CCZ26}, we have $\Tsb=\T^b$.
\end{proof}

{\bf Proof of Corollary \ref{mainapp}.} $(1)$ By Proposition \ref{FGD}(1) and Example \ref{Gexam}(1), the category $\D{\Lambda\Modcat}$ for a Gorenstein Artin algebra $\Lambda$ has finite strong finitistic dimension. Note that $\mathscr{K}^b(\Lambda\text{-}{\rm Proj})=\mathscr{D}^b(\Lambda\Modcat)$ if and only if $\Lambda$ has finite global dimension. Thus $(1)$ holds by applying Theorem \ref{main1}(2) to $\D{\Lambda\Modcat}$.

(2) By the last paragraph of the proof of Lemma \ref{lasttech}, a finite-dimensional noetherian scheme $X$ is regular if and only if $\mathscr{D}_{\mathrm{qc}}(X)^{sb}=\mathscr{D}_{\mathrm{qc}}(X)^b$.
Moreover, by Proposition \ref{FGD}(1) and Example \ref{Gexam}(2), the assumptions on $X$ in $(2)$ imply that the category $\mathscr{D}_{\mathrm{qc}}(X)$ has finite strong finitistic dimension. Now, $(2)$ holds by applying Theorem \ref{main1}(2) to $\mathscr{D}_{\mathrm{qc}}(X)$. \hfill$\square$

\medskip

In what follows, we explain how to use recollement to construct triangulated categories with finite strong finitistic dimension. The following lemma is useful.

\begin{lem}\label{lem-adjoint}
Let $\mathcal{S}$ and $\mathcal{T}$ be triangulated categories with compact generators
$G_{\mathcal{S}}$ and $G_{\mathcal{T}}$, respectively, and let $\mathbf{F} \colon \mathcal{S} \to \mathcal{T}$ be a triangle functor with a left adjoint $\mathbf{E} \colon \mathcal{T} \to \mathcal{S}$. Suppose that $\mathbf{F}$ preserves coproducts and compact objects, and $\mathbf{E}$ can be restricted to a functor $\Tsb\to\mathcal{S}^{sb}$. If $\mathcal{S}$ has finite strong finitistic dimension, then there exists an integer $m$ such that
$$
\mathbf{F}\mathbf{E}\big(\Tsb\cap{}^{\perp}(G_{\T}[1,+\infty))\big)\subseteq \bigcup_{n\le m}\overline{\langle G_{\mathcal{T}}\rangle}^{[n,\,m]}.
$$
If, in addition, $\mathbf{E}$ is fully faithful, then $\mathcal{T}$ has finite strong finitistic dimension.
\end{lem}
\begin{proof}
Since $\mathbf{F}$ preserves compact objects, $\mathbf{F}(G_{\mathcal{S}})\in\langle G_{\T}\rangle^{[a,\,b]}$ for some integers $a\leq b$. Then $\mathbf{F}(Y)\in\bigcup_{1\leq r}\langle G_{\T}\rangle^{[a+r,\,b+r]}$ for each $Y\in G_{\mathcal{S}}[1,+\infty)$.
Let $X\in\Tsb\cap{}^{\perp}(G_{\T}[1,+\infty))$. Then $\Hom_{\mathcal{T}}(X,\mathbf{F}(Y)[a])=0$.
Since $(\mathbf{E}, \mathbf{F})$ is an adjoint pair, we have  $$\Hom_{\mathcal{S}}(\mathbf{E}(X)[-a],Y)\simeq \Hom_{\mathcal{T}}(X[-a],\mathbf{F}(Y))\simeq \Hom_{\mathcal{T}}(X,\mathbf{F}(Y)[a])=0.$$
This implies $\mathbf{E}(X)[-a]\in{}^{\perp}(G_{\mathcal{S}}[1,+\infty))$.
Since $\mathbf{E}$ can be restricted to a functor $\Tsb\to\mathcal{S}^{sb}$, we have $\mathbf{E}(X)[-a]\in \mathcal{S}^{sb}$. As $\mathcal{S}$ has finite strong finitistic dimension, there exists an integer $t$ such that $\mathcal{S}^{sb}\cap{}^{\perp}(G_{\mathcal{S}}[1,+\infty))\subseteq \bigcup_{s\le t}\overline{\langle G_{\mathcal{S}}\rangle}^{[s,\,t]}$. Then $\mathbf{E}(X)[-a]\in \bigcup_{s\le t}\overline{\langle G_{\mathcal{S}}\rangle}^{[s,\,t]}$. Since $\mathbf{F}$ preserves coproducts, it follows that
$$\mathbf{F}\mathbf{E}(X)\in \bigcup_{s\le t}\overline{\langle \mathbf{F}(G_{\mathcal{S}}) \rangle}^{[s,\,t]}[a]\subseteq \bigcup_{s\le t}\overline{\langle G_{\mathcal{T}} \rangle}^{[s+a,\,t+b]}[a]\subseteq\bigcup_{n\le t+b-a}\overline{\langle G_{\mathcal{T}} \rangle}^{[n,\,t+b-a]}.$$
Then $\mathbf{F}\mathbf{E}(X)\in\bigcup_{n\le m}\overline{\langle G_{\mathcal{T}} \rangle}^{[n,\,m]}$ with $m:=t+b-a$. Thus the first assertion of Lemma \ref{lem-adjoint} holds. Suppose that $\mathbf{E}$ is fully faithful. Then $X\simeq \mathbf{F}\mathbf{E}(X)\in \bigcup_{n\le m}\overline{\langle G_{\mathcal{T}} \rangle}^{[n,\, m]}$. Thus $\mathcal{T}$ has finite strong finitistic dimension.
\end{proof}

\begin{prop}\label{construct1}
Let the following be a recollement of weakly approximable triangulated categories:
\begin{align*}
\xymatrixcolsep{4pc}\xymatrix{\mathcal{R} \ar[r]|{i_*=i_!} &\mathcal{S} \ar@<-2ex>[l]|{i^*} \ar@<2ex>[l]|{i^!} \ar[r]|{j^!=j^*}  &\mathcal{T}. \ar@<-2ex>[l]|{j_!} \ar@<2ex>[l]|{j_{*}}
}
\end{align*}
Suppose that $\mathcal{S}$ admits a bounded compact generator and $i_*$ preserves compact objects. If $\mathcal{R}$ and $\mathcal{T}$ have finite strong finitistic dimension, then so does $\mathcal{S}$.
\end{prop}
\begin{proof}
Let $G_{\mathcal{R}}$, $G_{\mathcal{S}}$ and $G_{\mathcal{T}}$ be compact generators of $\mathcal{R}$, $\mathcal{S}$ and $\mathcal{T}$, respectively. Then $G_{\mathcal{S}}$ is bounded by assumption.
Since $i_*(G_{\mathcal{R}})\in\mathcal{S}^c\subseteq \mathcal{S}^{sb}$, it follows from \cite[Theorem 3.7]{CCZ26} that the recollement induces a left recollement
\begin{align*}
\xymatrixcolsep{4pc}\xymatrix{
\mathcal{R}^{sb} \ar[r]|{i_*=i_!}
&\mathcal{S}^{sb} \ar@<-2ex>[l]|{i^*}   \ar[r]|{j^!=j^*}
&\mathcal{T}^{sb}. \ar@<-2ex>[l]|{j_!}
}
\end{align*}
Let $X\in\mathcal{S}^{sb}\cap{}^{\perp} (G_{\mathcal{S}}[1,+\infty)) $. Then there is a canonical triangle in $\mathcal{S}$: \[(\dag)\quad i_*i^*(X)[-1]\ra j_!j^!(X)\ra X\ra i_*i^*(X).\]
Applying Lemma \ref{lem-adjoint} to the adjoint pair $(i^*,i_*)$, we obtain $i_*i^*(X) \in \bigcup_{n\le r}\overline{\langle G_{\mathcal{S}}\rangle}^{[n,\,r]}\subseteq \mathcal{S}^{sb}$ for some $r\geq 0$.
Since $G_{\mathcal{S}}$ is bounded, there exists an integer $d\geq 0$ such that $\Hom_{\mathcal{S}}(G_{\mathcal{S}}[i],G_{{\mathcal{S}}})=0$ for any $i\geq d$.
Let $a\coloneqq r+d$. Then $G_{\mathcal{S}}(-\infty,r]\subseteq{}^{\perp}(G_{\mathcal{S}}[a,+\infty))$. Note that ${}^{\perp}(G_{\mathcal{S}}[a,+\infty))\subseteq \mathcal{S}$ is closed under direct summands, coproducts and extensions. Thus $$i_*i^*(X) \in \bigcup_{n\le r}\overline{\langle G_{\mathcal{S}}\rangle}^{[n,\,r]}\subseteq\mathcal{S}^{sb}\cap{}^{\perp}(G_{\mathcal{S}}[a,+\infty)).$$
By this and $X\in\mathcal{S}^{sb}\cap{}^{\perp} (G_{\mathcal{S}}[1,+\infty)) $, we see from the triangle ($\dag$) that $j_!j^!(X)\in\mathcal{S}^{sb}\cap{}^{\perp}(G_{\mathcal{S}}[a+1,+\infty))$.

By \cite[Theorem 5.1]{N96}, $j_!$ preserves compact objects. So, $j_!(G_{\T})\in\langle G_{\mathcal{S}}\rangle^{[b,\,c]}$ for some integers $b\leq c$. It follows that $j_!(Y)[b-a]\in\langle G_{\mathcal{S}}\rangle^{[a+1,\,+\infty)}$ for any $Y\in G_{\T}[1,+\infty)$. Since $j_!$ is fully faithful, we have
$$\Hom_{\mathcal{T}}(j^!(X)[a-b],Y)\simeq \Hom_{\mathcal{S}}(j_!j^!(X)[a-b],j_!(Y))\simeq \Hom_{\mathcal{S}}(j_!j^!(X),j_!(Y)[b-a])=0.$$
Thus $j^!(X)[a-b]\in\Tsb\cap{}^{\perp}(G_{\T}[1,+\infty))$. Further, there exists an integer $t$ such that $\Tsb\cap{}^{\perp}(G_{\T}[1,+\infty))\subseteq \bigcup_{n\le t}\overline{\langle G_{\mathcal{T}}\rangle}^{[n,\,t]}$ because $\mathcal{T}$ has finite strong finitistic dimension.  Thus $j^!(X)[a-b]\in \bigcup_{n\le t}\overline{\langle G_{\mathcal{T}}\rangle}^{[n,\,t]}$.
Since $j_!(G_{\mathcal{T}})\in \langle G_{\mathcal{S}}\rangle^{[b,\,c]}$ and $j_!$ preserves coproducts, $j_!j^!(X) \in \bigcup_{n\le t+a-b+c}\overline{\langle G_{\mathcal{S}}\rangle}^{[n,\,t+a-b+c]}$. By the  triangle ($\dag$), we have $X \in \bigcup_{n\le m}\overline{\langle G_{\mathcal{S}}\rangle}^{[n,\,m]}$ for some integer $m$. This shows $\mathcal{S}^{sb}\cap{}^{\perp} (G_{\mathcal{S}}[1,+\infty)) \subseteq \bigcup_{n\le m}\overline{\langle G_{\mathcal{S}}\rangle}^{[n,\,m]}$, and therefore $\mathcal{S}$ has finite strong finitistic dimension.
\end{proof}

\begin{cor}\label{construct2}

Let the following be a recollement of weakly approximable triangulated categories
\begin{align*}
\xymatrixcolsep{4pc}\xymatrix{\mathcal{R} \ar[r]|{i_*=i_!} &\mathcal{S} \ar@<-2ex>[l]|{i^*} \ar@<2ex>[l]|{i^!} \ar[r]|{j^!=j^*}  &\mathcal{T}. \ar@<-2ex>[l]|{j_!} \ar@<2ex>[l]|{j_{*}}
}
\end{align*}
Suppose that $i_*$ preserves compact objects, $\mathcal{S}$ admits a bounded compact generator, and both $\mathcal{R}$ and $\mathcal{T}$ are Gorenstein. The following statements are equivalent.

$(1)$ $\mathcal{S}^{sb}$ admits a bounded $t$-structure.

$(2)$ $\mathcal{S}^{sb}=\mathcal{S}^b$.

$(3)$ $\mathcal{R}^{sb}=\mathcal{R}^b$ and $\mathcal{T}^{sb}=\mathcal{T}^b$.
\end{cor}

\begin{proof}
Since $\mathcal{S}$ admits a bounded compact generator and $i_*$ preserves compact objects, both $\mathcal{R}$ and $\mathcal{T}$ admit bounded compact generators by \cite[Lemma~3.3]{CCZ26}.
By \cite[Theorem 5.4]{CCZ26}, $(2)$ and $(3)$ are equivalent, where the Gorenstein assumption on $\mathcal{R}$ and $\mathcal{T}$ is not needed.

Suppose that $\mathcal{R}$ and $\mathcal{T}$ are Gorenstein. By Proposition~\ref{FGD}(1), $\mathcal{R}$ and $\mathcal{T}$ have finite strong finitistic dimension. Then $(1)$ and $(2)$ are equivalent by Proposition \ref{construct1} and Theorem \ref{main1}(2).
\end{proof}

An application of Corollary \ref{construct2} to recollements of derived module categories is the following result.

\begin{cor}\label{construct3}
Let the following be a recollement of derived module categories of Artin algebras:
\begin{align*}
\xymatrixcolsep{4pc}\xymatrix{\D{\Lambda_1\Modcat} \ar[r]|{i_*=i_!} &\D{\Lambda_2\Modcat}\ar@<-2ex>[l]|{i^*} \ar@<2ex>[l]|{i^!} \ar[r]|{j^!=j^*}  &\D{\Lambda_3\Modcat}. \ar@<-2ex>[l]|{j_!} \ar@<2ex>[l]|{j_{*}}
}
\end{align*}
Suppose that $i_*(\Lambda_1)$ is compact in $\D{\Lambda_2\Modcat}$ and that both $\Lambda_1$ and $\Lambda_3$ are Gorenstein algebras. Then $\Kb{\Pmodcat{\Lambda_2}}$ admits a bounded $t$-structure if and only if $\Lambda_2$ has finite global dimension if and only if both $\Lambda_1$ and $\Lambda_3$ have finite global dimension.
\end{cor}

\begin{proof}
As shown in Example \ref{Gexam}(1), the derived module categories of Gorenstein Artin algebras are Gorenstein and weakly approximable. Note that $\Lambda_2$ is a bounded compact generator of $\D{\Lambda_2\Modcat}$.
Recall from \cite[Lemma 5.2(2)]{CCZ26} that, for a weakly approximable triangulated category $\T$ with a bounded compact generator, the equality $\Tsb = \T^{b}$ holds if and only if $\T$ has finite global dimension. In the case $\T=\mathscr{D}(S\Modcat)$ for an ordinary ring $S$, we see from \cite[Remark 5.3]{CCZ26} that $\T$ has finite global dimension if and only if $S$ has finite global dimension.
Thus Corollary \ref{construct3} follows from Corollary \ref{construct2}.
\end{proof}

In the literature, there are several methods for constructing recollements of derived module categories, especially via stratifying ideals of algebras (see \cite{CPS96}) or noncommutative tensor products (see \cite{CX19, CX21}).

Let $S$ be an arbitrary ring and $e\in S$ an idempotent (i.e., $e^2=e$). Recall from  \cite[Definition~2.1.1]{CPS96} that $SeS$ is called a \emph{stratifying ideal} of $S$ if the multiplication map $Se\otimes_{eSe}eS\to SeS$, sending $s_1e\otimes es_2$ to $s_1es_2$ for $s_1,s_2\in S$, is an isomorphism and $\Tor^{eSe}_i(Se,eS)=0$ for all $i>0$. This is also equivalent to saying that $\Tor_i^S(S/SeS, S/SeS)=0$ for all $i>0$. Thus if either ${_S}SeS$ or $SeS_S$ is flat, then $SeS$ is  stratifying.

\begin{cor}\label{app2}
Let $\Lambda$ be an Artin algebra and let $e\in\Lambda$ be an idempotent such that $\Lambda e\Lambda$ is a stratifying ideal of $\Lambda$. Suppose that

$(1)$ both $e\Lambda e$ and $\Lambda/\Lambda e\Lambda$ are Gorenstein algebras and

$(2)$ the $e\Lambda e$-module $e\Lambda$ has finite projective dimension (equivalently, the $\Lambda$-module $\Lambda e\Lambda$ has finite projective dimension).

Then $\Kb{\Pmodcat{\Lambda}}$ admits a bounded $t$-structure if and only if $\Lambda$ has finite global dimension.
\end{cor}
\begin{proof}
Since $\Lambda e\Lambda$ is a stratifying ideal of $\Lambda$,  there exists a recollement of derived categories
\begin{align*}
\xymatrixcolsep{4pc}\xymatrix{
  \D{(\Lambda/\Lambda e\Lambda)\Modcat}
  \ar[r]|(0.56){i_*=i_!}
  & \D{\Lambda\Modcat}
  \ar@<-2ex>[l]|(0.43){i^*}
  \ar@<2ex>[l]|(0.43){i^!}
  \ar[r]|(0.48){j^!=j^*}
  & \D{e\Lambda e\Modcat}
  \ar@<-2ex>[l]|(0.52){j_!}
  \ar@<2ex>[l]|(0.52){j_{*}}
}
\end{align*}\vspace*{-\belowdisplayskip}%

\noindent where $i_*$ is the derived restriction functor induced from the canonical surjection $\Lambda\to\Lambda/\Lambda e\Lambda$ and $j^*=e\Lambda \otimesL_{\Lambda}-$ (for example, see \cite[Example 4.5(4)]{AKL11}). Moreover, by \cite[Chapter IV, Proposition 1.11]{BR07}, $i_*$ preserves compact objects if and only if so does $j^*$. On the one hand, $i_*$ preserves compact objects if and only if $\Lambda/\Lambda e\Lambda\in \D{\Lambda\Modcat}^{c}$ if and only if $\Lambda e\Lambda\in \D{\Lambda\Modcat}^{c}$. On the other hand, $j^*$ preserves compact objects if and only if $e\Lambda\in \D{e\Lambda e\Modcat}^{c}$.
Thus ${_\Lambda}\Lambda e\Lambda$ has finite projective dimension if and only if
${_{e\Lambda e}}e\Lambda$ has finite projective dimension. Now, Corollary \ref{app2} follows from Corollary \ref{construct3}.
\end{proof}

\begin{rem}\label{Non-Gorenstein}
In Corollary \ref{app2}, the assumptions of $(1)$ and $(2)$ do not imply that $\Lambda$ is a Gorenstein algebra. A counterexample can be constructed from triangular matrix algebras.

Let $A$ and $B$ be Gorenstein Artin algebras, and let $M$ be a finitely generated $A$-$B$-bimodule.
We consider the triangular matrix algebra $\Lambda:= \bigl( \begin{smallmatrix} A & M \\ 0 & B \end{smallmatrix} \bigr)$ and  $e:= \bigl( \begin{smallmatrix} 0 & 0 \\ 0 & 1_{B} \end{smallmatrix} \bigr)$. Then $e\Lambda e \simeq B$ and $\Lambda / \Lambda e\Lambda \simeq A$ as algebras, and $\Lambda e\Lambda=\Lambda e$. In particular, ${}_\Lambda \Lambda e\Lambda$ is projective and ${_B}e\Lambda$ is projective. This means that $\Lambda$ satisfies all the assumptions of Corollary \ref{app2}. However, by \cite[Theorem~3.3]{C09}, $\Lambda$ is Gorenstein if and only if both ${}_{A}M$ and $M_{B}$ have finite projective dimension. Thus $\Lambda$ is not Gorenstein in general.
\end{rem}

\medskip

{\bf Acknowledgements.}
Jinbi Zhang was partially supported by the National Natural Science Foundation of China (Grant No.~12401038).

{\footnotesize
\smallskip
Hongxing Chen

School of Mathematical Sciences  \&  Academy for Multidisciplinary Studies, Capital Normal University, 100048
Beijing,

P. R. China;

{\tt Email: chenhx@cnu.edu.cn (H.X.Chen)}

\smallskip
Xiaohu Chen

School of Mathematical Sciences, Capital Normal University, 100048 Beijing, P. R. China;

{\tt Email: xiaohu.chen@cnu.edu.cn (X.H.Chen)}

\smallskip
Jinbi Zhang

School of Mathematical Sciences, Anhui University, 230601 Hefei, P. R. China;

{\tt Email: zhangjb@ahu.edu.cn (J.B.Zhang)}}

\end{document}